\documentclass[11pt,a4paper]{article}
\usepackage{indentfirst}
\usepackage{amsfonts}
\usepackage{amssymb}
\usepackage{mathrsfs}
\usepackage{amsmath}
\usepackage{amsthm}
\usepackage{enumerate}
\usepackage{cite}
\usepackage{mathrsfs}
\allowdisplaybreaks
\usepackage{geometry}
\usepackage{ifpdf}
\ifpdf
  \usepackage[colorlinks=true,linkcolor=blue,citecolor=red, final,backref=page,hyperindex]{hyperref}
\else
  \usepackage[colorlinks,final,backref=page,hyperindex,hypertex]{hyperref}
\fi

\usepackage{times}
\usepackage{indentfirst}
\usepackage{latexsym}
\usepackage[all]{xy}

\def\<{\langle}
\def\>{\rangle}

\def\c{\cdot}

\newtheorem{thm}{Theorem}[section]
\newtheorem{lem}[thm]{Lemma}
\newtheorem{cor}[thm]{Corollary}
\newtheorem{pro}[thm]{Proposition}
\newtheorem{ex}[thm]{Example}

\theoremstyle{definition}
\newtheorem{defi}{Definition}[section]

\theoremstyle{remark}
\newtheorem{rmk}{Remark}[section]
\begin{document}
\title{\bf $\mathcal{O}$-operators on Lie triple systems}\author{\bf T. Chtioui, A. Hajjaji, S. Mabrouk, A. Makhlouf}
\author{{ Taoufik Chtioui$^{1}$
 \footnote { E-mail: chtioui.taoufik@yahoo.fr}
,\  Atef Hajjaji$^{1}$
    \footnote { E-mail:  atefhajjaji100@gmail.com}
,\  Sami Mabrouk$^{2}$
 \footnote { E-mail: mabrouksami00@yahoo.fr }
\ and Abdenacer Makhlouf$^{3}$
 \footnote { E-mail: abdenacer.makhlouf@uha.fr $($Corresponding author$)$}
}\\
{\small 1.  University of Sfax, Faculty of Sciences,  BP
1171, 3038 Sfax, Tunisia} \\
{\small 2.  University of Gafsa, Faculty of Sciences, 2112 Gafsa, Tunisia}\\
{\small 3.~ IRIMAS - Département de Mathématiques, 18, rue des frères Lumière,
F-68093 Mulhouse, France}}
\date{}
\maketitle
\begin{abstract}
The purpose of this paper is to  study  cohomology and deformations of $\mathcal{O}$-operators on  Lie triple systems. We define a cohomology
of an $\mathcal{O}$-operator $T$ as the Lie-Yamaguti cohomology of a certain Lie triple system induced by $T$ with coefficients in
a suitable representation. Then we consider infinitesimal  and formal deformations of
$\mathcal{O}$-operators from cohomological viewpoint. Moreover we provide relationships between $\mathcal{O}$-operators on Lie algebras and associated Lie triple systems.
\end{abstract}

\textbf{Key words}:  Lie triple systems, $\mathcal{O}$-operator, representation, Lie-Yamaguti cohomology, deformation.

\textbf{Mathematics Subject Classification} (2020): 17A40, 17B56, 17B10, 17B38.

\numberwithin{equation}{section}

\tableofcontents

\section{Introduction}

The concept of Lie triple system   was introduced  first  by  Jacobson \cite{Jacobson}. The present formulation is due to Yamguti \cite{Yamaguti}. Moreover, it appeared in  Cartan's work on Riemannian Geometry  \cite{Cartan} and was strongly developed for Symmetric spaces and related spaces. Indeed,  the tangent space of a symmetric space is a Lie triple system. It turns out that  they have important applications in physics, in particular, in elementary particle theory and the theory of quantum mechanics, as well as numerical analysis of differential equations. They  have become an interesting subject in mathematics, their structure have been studied first by  Lister in \cite{Lister}.

The concept of $\mathcal{O}$-operators on Lie algebras appeared first, in a paper of
B. A. Kupershmidt \cite{Kupershmidt} as an operator analogue of classical $r$-matrices and Poisson structures. However,
Rota-Baxter operators were  introduced by G. Baxter \cite{Baxter} in his  study of fluctuation theory
in probability. Then developed by  G.-C. Rota \cite{Rota} in  Combinatorics. They have important applications in the algebraic aspects of the renormalization
in quantum field theory \cite{Connes}. Rota-Baxter operators on Lie algebras (resp. associative algebras)
are also related to the splitting  of algebraic structures.  In \cite{Rui}, The authors introduced the notion of a Rota-Baxter operator on a $3$-Lie algebra.  Recently, in order to construct solutions of the classical $3$-Lie
Yang-Baxter equation, the authors in \cite{Bai1} introduced a more general notion  called 
relative Rota-Baxter operator on a $3$-Lie algebra with respect to a representation .

 The  deformation theory using formal power series and suitable cohomologies was initiated by  Gerstenhaber
\cite{Gerstenhaber} for associative algebras. Nijenhuis and Richardson extended this study to Lie algebras \cite{Nijenhuis}. Deformations of  $3$-Lie algebras and Lie triple systems  were studied in \cite{Figueroa} and  \cite{Kubo,Zhang} respectively.
Recently deformations of certain operators, e.g. morphisms and Rota-Baxter operators
($\mathcal{O}$-operators) were deeply studied, see \cite{Makhlouf,Das,Fregier,Mandal,Tang}.
One needs a  cohomology to control deformations and extension problems of a given
algebraic structure. Cohomologies of various algebraic structures, such as associative
algebras, Lie algebras, Leibniz algebras, pre-Lie algebras and $n$-Lie algebras, are well known. Cohomology of Lie triple systems was introduced by K. Yamaguti in \cite{Yamaguti}. Recently, cohomology of Rota-Baxter operators ($\mathcal{O}$-operators) on Lie algebras and associative
algebras were developed in \cite{Tang,Das} respectively.  In \cite{Sheng},
the authors constructed a Lie $3$-algebra whose Maurer-Cartan elements are $\mathcal{O}$-operators on $3$-Lie algebras, and studied cohomologies and deformations of $\mathcal{O}$-operators on $3$-Lie algebras, see also \cite{chtioui}.
 
Inspired by these works, we tackle in this paper the cohomology theory of $\mathcal{O}$-operators on  Lie triple systems. We define a cohomology
of an $\mathcal{O}$-operator $T$ as the Lie-Yamaguti cohomology of a certain Lie triple system induced by $T$ with coefficients in
a suitable representation. Moreover, we provide  some connections between $\mathcal{O}$-operators on Lie algebras and Lie triple systems.
 In Section 2, we summarize som basics and briefly recall representations and cohomology
of Lie triple systems. In Section 3, we investigate  $\mathcal{O}$-operator and Nijenhuis operator on Lie triple systems. In Section 4, we define the cohomology of $\mathcal{O}$-operator on Lie triple system using the underlying Lie triple systems  of the $\mathcal{O}$-operator. Section 5 deals with one-parameter deformations of
$\mathcal{O}$-operators using the cohomology theory established in Section 4. Finally, in Section 6, we describe some connections between cohomology of  $\mathcal{O}$-operators on Lie algebras and associated Lie triple systems.

In this paper all vector spaces are considered over a field $\mathbb{F}$  of characteristic $0$.
\section{Preliminaries}
In this section, we recall   representations and cohomology theory  of Lie triple systems.
\begin{defi}
A Lie triple systems (L.t.s) is a vector space $L$ endowed with a ternary bracket $[\cdot,\cdot,\cdot]: \wedge^2L \otimes L \to L$ satisfying:
\begin{align}
    & [x,y,z]+[y,z,x]+[z,x,y]=0, \label{lts 1}
    \\
    & [x,y,[z,t,e]]=[[x,y,z],t,e]+[z,[x,y,t],e]+[z,t,[x,y,e]], \label{lts 2}
\end{align}
for any $x,y,z,t,e \in L$.
\begin{ex}\label{lts}
Let $(L,[\c,\c])$ be a Lie algebra. We define $[\c,\c,\c]:(\wedge^{2}L)\otimes L\rightarrow L$ by
\begin{equation}
    [x,y,z]=[[x,y],z],\quad \forall x,y,z \in L.
\end{equation}
Then $(L,[\c,\c,\c])$  becomes a L.t.s naturally.
\end{ex}
\end{defi}
A morphism $f : (L, [\c, \c, \c]) \rightarrow (L^{'}
, [\c, \c, \c]^{'})$ of L.t.s is a linear map satisfying \begin{equation*}
    f([x, y, z]) = [f(x), f(y), f(z)]^{'},\quad \forall x,y,z\in L.
\end{equation*}
An isomorphism is a bijective morphism.\\

Denote by $X = (x, y)\in L\otimes L$ and $\mathcal{L}(X):L \rightarrow L$ defined by $\mathcal{L}(X)z_i = [x, y, z_i]$, then the identity \eqref{lts 2} can be rewritten
in the form
\begin{equation*}
    \mathcal{L}(X)[z_1, z_2, z_3] = [\mathcal{L}(X)z_1, z_2, z_3] + [z_1, \mathcal{L}(X)z_2, z_3] + [z_1, z_2, \mathcal{L}(X)z_3],
\end{equation*}
which means that $\mathcal{L}(X)$ is a derivation with respect to the bracket $[\c, \c,\c $].

In \cite{Yamaguti},  K. Yamaguti introduced the notion of representation and cohomology theory of L.t.s. Later, the authors in \cite{Harris,Hodge} and \cite{Kubo} studied the cohomology theory of L.t.s from a different point of view. K. Yamaguti’s
work can be described as follows.
\begin{defi}
A representation of a L.t.s $L$ on a vector space $V$ is a bilinear map $\theta: \otimes^{2} L \rightarrow End(V)$, such that the following conditions are satisfied:
\begin{align}
    & \theta(z,t)\theta(x,y)-\theta(y,t)\theta(x,z)-\theta(x,[y,z,t])+D(y,z)\theta(x,t)=0, \label{rep lts 1}
    \\
    &\theta(z,t)D(x,y)-D(x,y)\theta(z,t)+\theta([x,y,z],t)+\theta(z,[x,y,t])=0,  \label{rep lts 2}
\end{align}
where $D(x,y)=\theta(y,x)-\theta(x,y)$, for all $x,y,z,t\in L$. We denote by 
$(L, [\cdot,\cdot, \cdot];\theta)$, the \textsf{\textsf{L.t.sRep}} pair.
\end{defi}

\begin{rmk}
Let $L$ be a L.t.s and $\theta(x,y)$ be the linear map $z \mapsto [z, x, y]$ of $L$ into itself for all $x, y \in L$. Using \eqref{lts 2}, then
we can prove that $L$ is an $L$-module. Moreover, by \eqref{lts 1}, $D(x, y)$ becomes a linear map $z \mapsto[x, y, z]$ (inner derivation). In this paper, we define $\theta (x, y):=\mathcal{R}(x, y), D(x, y):=\mathcal{L}(x, y)$ where $\mathcal{L}(x,y)(z) =
[x, y, z],\; \mathcal{R}(x, y)(z) = [z, x, y]$.
\end{rmk}
As usual, we have a characterisation of a representation by a semi-direct product i.e. $(V,\theta)$ is a representation of a L.t.s $L$ if and only if $L\oplus V$ is a L.t.s under the following bracket
\begin{align}\label{semidirect product}
[x+u,y+v,z+w]_{L\oplus V}=[x,y,z]+\theta(y,z)u-\theta(x,z)v+D(x,y)w,
\end{align}
for any $x,y,z \in L$ and $u,v,w \in V$.

Let
$(L, [\cdot,\cdot, \cdot];\theta)$ be a \textsf{L.t.sRep} pair. For each $n \geqslant 0$, we denote by $C^{2n+1}(L,V)$, the vector space of $(2n+1)$-cochains of $L$ with coefficients in $V$: $f\in C^{2n+1}(L,V)$ is a multilinear function of $\times^{2n+1} L$ into $V$ satisfying

$$f(x_1,x_2,\cdots,x_{2n-2},x,x,y)=0,$$
and
$$f(x_1, x_2, \cdots, x_{2n-2}, x, y, z) + f(x_1, x_2, \cdots, x_{2n-2}, y, z, x) + f(x_1, x_2,\cdots, x_{2n-2}, z, x, y) = 0.$$

The Yamaguti coboundary operator $\delta^{2n-1}: C^{2n-1}(L,V) \rightarrow C^{2n+1}(L,V) $ is defined by

\begin{align}
  &  \delta^{2n-1} f(x_1,x_2, \cdots , x_{2n+1})\nonumber\\
    =&\theta(x_{2n},x_{2n+1})f(x_1,x_2, \cdots , x_{2n-1})- \theta(x_{2n-1},x_{2n+1})f(x_1,x_2, \cdots , x_{2n-2},x_{2n}) \nonumber \\
    +& \sum_{k=1}^n (-1)^{n+k}D(x_{2k-1},x_{2k})f(x_1,x_2, \cdots , \widehat{x}_{2k-1},\widehat{x}_{2k}, \cdots , x_{2n+1}) \nonumber \\
    +& \sum_{k=1}^n \sum_{j=2k+1}^{2n+1}  (-1)^{n+k+1} f(x_1,x_2, \cdots,  \widehat{x}_{2k-1},\widehat{x}_{2k}, \cdots, [x_{2k-1},x_{2k},x_j],\cdots , x_{2n+1})
\end{align}
for all $f \in C^{2n-1}(L,V),\ n\geqslant 1$, where $\;\widehat{}\;$ denotes omission. The Yamaguti cochain forms a complex with this coboundary as follow
$$ C^{1}(L,V)\overset{\delta^{1}}{\longrightarrow} C^{3}(L,V)\overset{\delta^{3}}{\longrightarrow}  C^{5}(L,V) \longrightarrow \cdots,$$
and $\delta^{2n+1}\circ \delta^{2n-1}=0\;for\;n=1,2,\cdots$ ( see \cite{Yamaguti} for more details). Hence we get the
Yamaguti cohomology group  $H^{\bullet}(L , V ) = Z^{\bullet}(L , V )/B^{\bullet}(L , V )$, where $Z^{\bullet}(L , V )$ is the space of
cocycles and $B^{\bullet}(L , V )$ is the space of coboundaries. 
\begin{defi}\cite{Zhang}\label{cocycles}
Let $(V,\theta)$ be a representation of a L.t.s $L$.
\begin{enumerate}
    \item 
A linear map $f\in C^{1}(L, V )$
is a $1$-cocycle (closed) if
\begin{equation}
    D(x_1,x_2)f(x_3)-\theta(x_1,x_3)f(x_2)+\theta(x_2,x_3)f(x_1)-f([x_1,x_2,x_3])=0,
\end{equation}
and a map $f\in C^{3}(L,V)$ is a $3$-coboundary if there exists a map $g \in C^{1}(L, V )$ such that $f=\delta^{1}g$.
\item A map $f\in C^{3}(L,V)$ is called a $3$-cocycle if $\forall x_1,x_2,x_3,y_1,y_2,y_3\in L$,
\begin{align}
  \label{SkewCochain}  &f(x_1,x_1,x_2)=0,\\
    &f(x_1,x_2,x_3)+f(x_2,x_3,x_1)+f(x_3,x_1,x_2)=0,\\
    &f(x_1,x_2,[y_1,y_2,y_3])+D(x_1,x_2)f(y_1,y_2,y_3)\nonumber\\
    &=f([x_1,x_2,y_1],y_2,y_3)+f(y_1,[x_1,x_2,y_2],y_3)+f(y_1,y_2,[x_1,x_2,y_3])\nonumber\\
    &+\theta(y_2,y_3)f(x_1,x_2,y_1)-\theta(y_1,y_3)f(x_1,x_2,y_2)+D(y_1,y_2)f(x_1,x_2,y_3).
\end{align}
\end{enumerate}

\end{defi}

\section{$\mathcal{O}$-operators on Lie triple systems}
In this section, we recall the notion of pre-Lie triple system structure introduced in \cite{Mabrouk} and define Rota-Baxter operators, $\mathcal{O}$-operators on L.t.s, then 
we study their morphisms. Moreover, we give some characterisations of $\mathcal{O}$-operators in terms of  Nijenhuis operators and graphs. 
\begin{defi}
Let $L$ be a L.t.s.  A linear map $R : L \rightarrow L$ is said to be a
Rota-Baxter operator of weight $0$ if it satisfies
\begin{equation*}
    [R(x),R(y),R(z)]=R\Big([R(x),R(y),z]+[R(x),y,R(z)]+[x,R(y),R(z)]\Big),\;\forall x,y,z \in L.
\end{equation*}
\end{defi}
The notion of $\mathcal{O}$-operator (also called relative Rota-Baxter
operator or Kupershmidt operator) is a generalization of Rota-Baxter operators in the presence of arbitrary representation.
\begin{defi}
Let
$(L, [\cdot,\cdot, \cdot];\theta)$ be a \textsf{L.t.sRep} pair.  A linear map $T: V \to L$ is called an $\mathcal{O}$-operator of $L$ with respect to $\theta $ if it satisfies
\begin{align}\label{O op on lts}
[Tu,Tv,Tw]=T\Big(D(Tu,Tv)w+\theta(Tv,Tw)u-\theta(Tu,Tw)v\Big), \quad \forall u,v,w \in V.
\end{align}

If $V=L$, then $T$ is a Rota-Baxter operator of weight $0$ on $L$  with respect to the adjoint representation. Thus  $\mathcal{O}$-operators are  generalization of Rota-Baxter operators.
\end{defi}
Now we give the following characterization of an $\mathcal{O}$-operator in terms of graphs.
\begin{pro}\label{graph}
A linear map $T:V\rightarrow L$ is an $\mathcal{O}$-operator if and only if the graph $Gr(T)=\{(Tu,u)|\;u\in V\}$ is a subalgebra of the
semi-direct product $L \oplus V$.
\end{pro}
\begin{proof}
Let $(Tu,u)$, $(Tv,v)$ and $(Tw,w)$ $\in Gr(T)$. Then we have
\begin{align*}
&\quad \;[Tu+u,Tv+v,Tw+w]_{L \oplus V}
=[Tu,Tv,Tw]+ \theta(Tv,Tw)u-\theta(Tu,Tw)v+D(Tu,Tv)w.
\end{align*}
 Assume that  $Gr(T)$ is a subalgebra of the
semi-direct product $L\oplus V$, then  we have
\begin{align*}
&[Tu,Tv,Tw]=
T\Big(   \theta(Tv,Tw)u-\theta(Tu,Tw)v+D(Tu,Tv)w  \Big).
\end{align*}
On the other hand, if $T$ is an $\mathcal{O}$-operator,  we obtain
\begin{align*}
\quad \;[Tu+u,Tv+v,Tw+w]_{L \oplus V}
&=T\Big(   \theta(Tv,Tw)u-\theta(Tu,Tw)v+D(Tu,Tv)w  \Big) \\
 &+   \theta(Tv,Tw)u-\theta(Tu,Tw)v+D(Tu,Tv)w
 \in Gr(T).
\end{align*}
Hence $Gr(T)$ is a subalgebra of the
semi-direct product $L\oplus V$.
\end{proof}
In the following proposition, we show that an $\mathcal{O}$-operator can be lifted up to a
Rota-Baxter operator.
\begin{pro}
Let 
$(L, [\cdot,\cdot, \cdot];\theta)$ be a \textsf{L.t.sRep} pair and
$T : V \rightarrow L$ be a linear map. Define $\widehat{T} \in  End(L\oplus V)$ by $\widehat{T}(x, u) = (T(u), 0)$. Then $T$ is an
$\mathcal{O}$-operator if and only $\widehat{T}$ is a Rota-Baxter operator on $L\oplus V$.
\end{pro}
\begin{proof}
For any $x,y,z\in L$ and $u,v,w\in V$, we have
\begin{align*}
    &[\widehat{T}(x,u),\widehat{T}(y,v),\widehat{T}(z,w)]_{L\oplus V}\\
    &-\widehat{T}\Big([\widehat{T}(x,u),\widehat{T}(y,v),(z,w)]_{L\oplus V}+[\widehat{T}(x,u),(y,v),\widehat{T}(z,w)]_{L\oplus V}+[(x,u),\widehat{T}(y,v),\widehat{T}(z,w)]_{L\oplus V}\\
   = &[(Tu,0),(Tv,0),(Tw,0)]_{L\oplus V}\\
    &-\widehat{T}\Big([(Tu,0),(Tv,0),(z,w)]_{L\oplus V}+[(Tu,0),(y,v),(Tw,0)]_{L\oplus V}+[(x,u),(Tv,0),(Tw,0)]_{L\oplus V}\\
    =&\Big([Tu,Tv,Tw],0\Big)-\Big(T(   \theta(Tv,Tw)u-\theta(Tu,Tw)v+D(Tu,Tv)w),0  \Big)\\
    =&\Big([Tu,Tv,Tw]-T(   \theta(Tv,Tw)u-\theta(Tu,Tw)v+D(Tu,Tv)w),0\Big).
\end{align*}
Then $\widehat{T}$ is a Rota-Baxter operator on $L \oplus V$ if and only if $T$ is an $\mathcal{O}$-operator.
\end{proof}
\begin{ex}
{\rm Let $L$ be a 2-dimensional L.t.s with a basis $\{e_1,e_2\}$ and a bracket defined by}
$$[e_1,e_2,e_2]=e_1.$$
{\rm The operator}
$
T=\begin{pmatrix}
 0 & a \\
 0 & b
 \end{pmatrix}
$
{\rm is a Rota-Baxter operator on $L$.}
\end{ex}

\begin{ex}
{\rm Let $L$ be a 4-dimensional L.t.s with a basis $\{e_1,e_2,e_3,e_4\}$ and a bracket defined by}
$$[e_1,e_2,e_1]=e_4.$$
{\rm Then}
\begin{gather*}
T=\begin{pmatrix}
 0 & a &0 & 0\\
 0 & 0 & 0& 0\\
 b & c & d & e\\
 f & g & h & k
 \end{pmatrix}
\end{gather*}
{\rm  defines a Rota-Baxter operator on $L$.}
\end{ex}

Another characterization of an $\mathcal{O}$-operator  can be given in terms of Nijenhuis operators on L.t.s. First of all, we give the following definition of 
 Nijenhuis operator
 on a L.t.s $L$.
 \begin{defi}
 A linear map $N: L \to L$ is called Nijenhuis operator on a L.t.s $L$ if $N$ satisfies the following condition
\begin{align}\label{Nij op}
[Nx,Ny,Nz]=& N\Big([Nx,Ny,z]+[x,Ny,Nz]+[Nx,y,Nz] \nonumber\\
-& N\Big([Nx,y,z]+[x,Ny,z]+[x,y,Nz]  
-N [x,y,z]\Big)\Big).
\end{align}
\end{defi}
\begin{rmk}
The above definition of Nijenhuis operator on a L.t.s is similar to the definition of Nijenhuis operator on $3$-Lie algebras introduced in \cite{Liu} in the study of $2$-order
trivial deformations. On the other hand, in \cite{Zhang}, the author has introduced another notion of a Nijenhuis operator on a L.t.s in the study of $1$-order
trivial deformations. In that definition, there is a quite strong condition $[Nx_1, Nx_2, Nx_3] = 0$,
whereas the above definition for $n = 3$ is the Eq.\eqref{Nij op}. So obviously, the
above definition is different from  the definition given in \cite{Zhang} and we think that is the right definition.
\end{rmk}
\begin{lem}
If $N$ is a Nijenhuis operator on $L$. Then $(L,[\cdot,\cdot,\cdot]_N)$ is a L.t.s, where
\begin{align}\label{nijstructure}
[x,y,z]_N=& [Nx,Ny,z]+[x,Ny,Nz]+[Nx,y,Nz]\nonumber\\
-& N\Big([Nx,y,z]+[x,Ny,z]+[x,y,Nz]
 -N [x,y,z] \Big),
\end{align}
and $N$ is a homomorphism from $(L,[\cdot,\cdot,\cdot]_N)$ to $(L,[\cdot,\cdot,\cdot])$.
\end{lem}
\begin{proof}
It follows from straightforward computations.
\end{proof}
The following result is also  straightforward, so we omit details.
\begin{pro}
Let $(L, [\cdot,\cdot, \cdot];\theta)$ be a \textsf{L.t.sRep} pair. A linear operator $T:V\rightarrow L$ is an $\mathcal{O}$-operator if and only if \begin{equation*}
     \overline{T}=\begin{pmatrix}
                                                   0 & T \\
                                                   0 & 0 \\
                                                 \end{pmatrix}:L\oplus V\rightarrow L\oplus V
\end{equation*}
is a Nijenhuis operator on the the semi-direct product L.t.s $L\oplus V$.
\end{pro}

Next, we recall the notion of pre-Lie triple systems introduced in \cite{Mabrouk} which is the structure induced from $\mathcal{O}$-operators. 
\begin{defi}
Let $L$ be a vector space with a $3$-linear map $\{\c,\c, \c\} :  \otimes^{3} L  \rightarrow L$. The pair $(L, \{\c,\c,\c\})$ is called a pre-Lie triple system if the following identities holds
\begin{align}
    &\label{cond1}\{x_5, x_1, [x_2, x_3, x_4]_C\} = \{\{x_5, x_1, x_2\}, x_3, x_4\} - \{\{x_5, x_1, x_3\}, x_2, x_4\}+ \{x_2, x_3, \{x_5, x_1, x_4\}\}^{*},\\
    &\label{cond2}\{x_1, x_2, \{x_5, x_3, x_4\}\}^{*} \;= \{\{x_1, x_2, x_5\}^{*}, x_3, x_4\} + \{x_5, [x_1, x_2, x_3]_C, x_4\}
+ \{x_5, x_3, [x_1, x_2, x_4]_C\}, 
\end{align}
where $\{\c,\c,\c\}^{*}$ and $[\c,\c,\c]_C$ are defined by
\begin{align}
    &\{x,y,z\}^{*}=\{z,y,x\}-\{z,x,y\},\\
    &\label{comm}[x,y,z]_C\;=\{z,y,x\}-\{z,x,y\}+\{x,y,z\}-\{y,x,z\}\nonumber\\
    &\quad \quad\quad \;\;\;\; =\{x,y,z\}^{*}+\{x,y,z\}-\{y,x,z\},
\end{align}
for any $x,y,z,x_i \in L$.
\end{defi}
It follows from \eqref{cond1} and \eqref{cond2} that the new operation $[\c,\c,\c]_C :\otimes^{3}L\rightarrow L$ defined by Eq. \eqref{comm}
turns out to be a L.t.s.

An $\mathcal{O}$-operator has an underlying pre-Lie triple system structure. 
\begin{pro}
Let
$(L, [\cdot,\cdot, \cdot];\theta)$ be a \textsf{L.t.sRep} pair. Suppose that the linear map $T : V \rightarrow L$ is an $\mathcal{O}$-operator associated to $(V, \theta)$. Then there
exists a pre-Lie triple system structure on $V$ given by
\begin{equation}
    \{u,v,w\}=\theta(Tv,Tw)u,\quad \forall u,v,w \in V.
\end{equation}
\end{pro}
Next, we study morphism between O-operators.
\begin{defi}
Let $T$ be an $\mathcal{O}$-operator on a \textsf{L.t.sRep} pair
$(L, [\cdot,\cdot, \cdot];\theta)$. Suppose $(L^{'}, [\cdot,\cdot, \cdot]^{'};\theta^{'})$ is another \textsf{L.t.sRep} pair. Let $T^{'}:V^{'}\rightarrow L^{'} $ be an $\mathcal{O}$-operator. A morphism of $\mathcal{O}$-operators from  $T$ to $T^{'}$
consists of a pair $(\phi,\psi)$ of a L.t.s morphism $\phi : L \rightarrow L^{'}$ and a linear map $\psi: V \rightarrow V^{'}$ satisfying
\begin{align}
    &\label{c1}\phi \circ T = T^{'}\circ \psi,\\
    &\label{c2}\psi\theta(x,y)=\theta(\phi(x),\phi(y))\psi,\quad \forall x,y \in L.
\end{align}
It is called an isomorphism if $\phi$ and $\psi$ are both bijective.  
\end{defi}
The proof of the following result is straightforward then we omit the details.
\begin{pro}
A pair of linear maps $(\phi : L \rightarrow L^{'}, \psi : V \rightarrow V^{'})$ is a morphism of $\mathcal{O}$-operators
from  $T$ to $T^{'}$ if and only if
\begin{equation}
    Gr((\phi,\psi))=\Big\{\Big((x,u),(\phi(x),\psi(u)\Big)\;|\; x\in L,u\in V \Big\} \subset (L \oplus V)\oplus (L^{'}\oplus V^{'}),
\end{equation}
is a subalgebra, where $L \oplus V$ and $L^{'} \oplus V^{'}$ are equipped with semi-direct product structures of L.t.s.
\end{pro}
In the rest of the paper, we will be most interested in morphism between $\mathcal{O}$-operators on the
same L.t.s with respect to the same module.
\begin{pro}
 Let $T$ and $T^{'}$ be two $\mathcal{O}$-operators on the \textsf{L.t.sRep} pair
$(L, [\cdot,\cdot, \cdot];\theta)$. If $(\phi, \psi)$ is a morphism (resp. an isomorphism) from  $T$ to $T^{'}$, then $\psi : V \rightarrow V$
is a morphism (resp. an isomorphism) between induced pre-Lie triple system structures.
\end{pro}
\begin{proof}
For all $u,v,w \in V$, by Eqs \eqref{c1} and \eqref{c2} we have 
\begin{align*}
    &\psi(\{u,v,w\}_{T})=\psi(\theta(T(v),T(w))u)=\theta(\phi(T(v)),\phi(T(w)))\psi(u)\\
    &\quad \quad \quad \quad \quad \;\;\; =\theta(T^{'}(\psi(v)),T^{'}(\psi(w)))\psi(u)=\{\psi(u),\psi(v),\psi(w)\}_{T^{'}}.
\end{align*}
Hence the result follows.
\end{proof}

\section{Cohomology of $\mathcal{O}$-operators on Lie triple systems}
In this section, we define a cohomology of an $\mathcal{O}$-operator $T$ on a L.t.s $L$  as the Lie-Yamaguti cohomology
of a certain L.t.s  with coefficients in a suitable
representation on $L$. This cohomology will be used further, to study deformations of $T$.

It was proved in \cite{Mabrouk} that there is a pre-Lie triple system structure on $V$ as the underlying
structure of an $\mathcal{O}$-operator on a L.t.s. Consequently, there is a
subadjacent L.t.s structure on $V$ given as follows.
\begin{lem}\label{struV}
Let $T: V \to L$ be an $\mathcal{O}$-operator on a \textsf{L.t.sRep} pair
$(L, [\cdot,\cdot, \cdot];\theta)$.  Then $(V,[\cdot,\cdot,\cdot]_T)$ is a L.t.s, where
\begin{align}\label{crochet de V}
   [u,v,w]_T:=D(Tu,Tv)w+\theta(Tv,Tw)u-\theta(Tu,Tw)v.
\end{align}
Moreover, $T$ is a homomorphism of L.t.s i.e. $T([u,v,w]_T)=[Tu,Tv,Tw]$.
\end{lem}
Furthermore, there is a representation of the above L.t.s  $(V, [\c,\c,\c]_T )$ on $L$:
\begin{pro}
Let $T$ be an $\mathcal{O}$-operator on a \textsf{L.t.sRep} pair
$(L, [\cdot,\cdot, \cdot];\theta)$. Define $\theta_T:\otimes^{2}V\rightarrow gl(L)$ by
\begin{equation}\label{indurep}
\theta_T(u,v)x=[x,Tu,Tv]+T\Big(\theta(x,Tv)u-D(x,Tu)v \Big),\ \forall x \in L,\ u,v \in V.
\end{equation}
Then $(L,\theta_T)$ is a representation of the L.t.s $( V, [\cdot,\cdot, \cdot]_T)$ on $L$.
\end{pro}

\begin{proof}
One can show it directly by a tedious computation. Here we take a different approach
using Nijenhuis operators on L.t.s. Let $T$ be an $\mathcal{O}$-operator on a \textsf{L.t.sRep} pair
$(L, [\cdot,\cdot, \cdot];\theta)$. We define $\overline{T} : L\oplus V \rightarrow L\oplus V$ by
\begin{equation*}
  \overline{T}(x+u)=T(u),\quad \forall x\in L,\; u\in V.  
\end{equation*}
Then $\overline{T}$ is a Nijenhuis operator on the semidirect product L.t.s $L\ltimes_{\theta}V$ and $\overline{T}\circ \overline{T} = 0$.
Then by Eq. \eqref{nijstructure}, there is a L.t.s structure $[\c,\c,\c]_{\overline{T}}$ on the vector space $L\oplus V$ given by
\begin{align*}
    &\quad \;[x+u,y+v,z+v]_{\overline{T}}\\
    &=[\overline{T}(x+u),\overline{T}(y+v),z+w]_{L\oplus V}+[\overline{T}(x+u),y+v,\overline{T}(z+w)]_{L\oplus V}+[x+u,\overline{T}(y+v),\overline{T}(z+w)]_{L\oplus V}\\
    &\quad \;-\overline{T}\Big([\overline{T}(x+u),y+v,z+w]_{L\oplus V}+[x+u,\overline{T}(y+v),z+w]_{L\oplus V}+[x+u,y+v,\overline{T}(z+w)]_{L\oplus V}\Big)\\
    &=[Tu,Tv,z+w]_{L\oplus V}+[Tu,y+v,Tw]_{L \oplus V}+[x+u,Tv,Tw]_{L \oplus V}\\
    &\quad \;-\overline{T}\Big([Tu,y+v,z+w]_{L\oplus V}+[x+u,Tv,z+w]_{L\oplus V}+[x+u,y+v,Tw]_{L\oplus V}\Big)\\
    &=[Tu,Tv,z] +D(Tu,Tv)w+[Tu,y,Tw]-\theta(Tu,Tw)v+[x,Tv,Tw]+\theta(Tv,Tw)u\\
    &\quad \;-\overline{T}\Big([Tu,y,z]-\theta(Tu,z)v+D(Tu,y)w+[x,Tv,z]+\theta(Tv,z)u+D(x,Tv)w\\
    &\quad \;+[x,y,Tw]+\theta(y,Tw)u-\theta(x,Tw)v\Big)\\
    &=[Tu,Tv,z] +D(Tu,Tv)w+[Tu,y,Tw]-\theta(Tu,Tw)v+[x,Tv,Tw]+\theta(Tv,Tw)u\\
    &\quad \;-T\Big(-\theta(Tu,z)v+D(Tu,y)w+\theta(Tv,z)u+D(x,Tv)w+\theta(y,Tw)u-\theta(x,Tw)v\Big),
\end{align*}
or by Eq. \eqref{lts 1}, we have
\begin{align*}
    &[Tu,Tv,z]=-[Tv,z,Tu]-[z,Tu,Tv]= [z,Tv,Tu]-[z,Tu,Tv].
\end{align*}
Then we obtain that 
\begin{align*}
    &D_T(u,v)z=\theta_T(v,u)z-\theta_T(u,v)z\\
    &\quad \quad \quad\;\; \;\;\;=[z,Tv,Tu]+T\Big(\theta(z,Tu)v-D(z,Tv)u\Big)\\
    &\quad \quad \quad\; \;\;\;\;-[z,Tu,Tv]-T\Big(\theta(z,Tv)u-D(z,Tu)v\Big)\\
    &\quad \quad \quad \;\;\;\;\;=[Tu,Tv,z]+T\Big(\theta(z,Tu)v-(\theta(Tv,z)u-\theta(z,Tv)u)\Big)\\
    &\quad \quad \quad \;\;\;-T\Big(\theta(z,Tv)u-(\theta(Tu,z)v-\theta(z,Tu)v)\Big)\\
    &\quad \quad \quad \;\;\;=[Tu,Tv,z]-T\Big(-\theta(Tu,z)v+\theta(Tv,z)u\Big).
    \end{align*}
Finally, we have 
\begin{equation}
   [x+u,y+v,z+v]_{\overline{T}}=[u,v,w]_T+\theta_T(v,w)x-\theta_T(u,w)y+D_T(u,v)z. 
\end{equation}
Since a semidirect product of a L.t.s is equivalent to a representation of a L.t.s,
 we deduce that $(L,\theta_T )$ is a representation of the L.t.s $(V, [\c,\c,\c]_T )$.
\end{proof}

Let
$(V, [\cdot,\cdot, \cdot]_T;\theta_T)$ be a \textsf{L.t.sRep} pair. For each $n \geqslant 0$, we denote by $C^{2n+1}(V,L)$, the vector space of $(2n+1)$-cochains of $V$ with coefficients in $L$. An element  $f\in C^{2n+1}(V,L)$ is a multilinear map  on $\times^{2n+1}  V$ into $L$ satisfying

$$f(v_1,v_2,\cdots,v_{2n-2},v,v,u)=0,$$
and
$$f(v_1, v_2, \cdots, v_{2n-2}, v, u, w) + f(v_1, v_2, \cdots, v_{2n-2}, u, w, v) + f(v_1, v_2,\cdots, v_{2n-2}, w, v, u) = 0.$$

Let $\delta^{2n-1}_T: C^{2n-1}(V,L) \to C^{2n+1}(V,L)$ be the corresponding coboundary operator on the L.t.s $(V,[\cdot,\cdot,\cdot]_T)$ with coefficients in the representation $(L,\theta_T)$.  More precisely, $\delta_T: C^{2n-1}(V,L) \to C^{2n+1}(V,L)$ is given by
\begin{align}
  &  \delta_T f(v_1,v_2, \cdots , v_{2n+1})\nonumber\\
    =&\theta_T(v_{2n},v_{2n+1})f(v_1,v_2, \cdots , v_{2n-1})- \theta_T(v_{2n-1},v_{2n+1})f(v_1,v_2, \cdots , v_{2n-2},v_{2n}) \nonumber \\
    &+ \sum_{k=1}^n (-1)^{n+k}D_T(v_{2k-1},v_{2k})f(v_1,v_2, \cdots , \widehat{v}_{2k-1},\widehat{v}_{2k}, \cdots , v_{2n+1}) \nonumber \\
    &+ \sum_{k=1}^n \sum_{j=2k+1}^{2n+1}  (-1)^{n+k+1} f(v_1,v_2, \cdots,  \widehat{v}_{2k-1},\widehat{v}_{2k}, \cdots, [v_{2k-1},v_{2k},v_j]_T,\cdots , v_{2n+1}), 
\end{align}

for all $f \in C^{2n-1}(V,L),\ n\geqslant 1$, where $\;\widehat{}\;$ denotes omission. With this coboundary operator  the Yamaguti cochain forms a complex
$$ C^{1}(V,L)\overset{\delta_T^{1}}{\longrightarrow} C^{3}(V,L)\overset{\delta_T^{3}}{\longrightarrow}  C^{5}(V,L) \longrightarrow \cdots,$$
and $\delta_T^{2n+1}\circ \delta_T^{2n-1}=0\;for\;n=1,2,\cdots$.

\begin{defi}
Let
$(V, [\cdot,\cdot, \cdot]_T;\theta_T)$ be a \textsf{L.t.sRep} pair.
\begin{enumerate}
    \item 
A linear map $f\in C^{1}(V, L )$
is called $1$-cocycle if
\begin{align}
    &D_T(v_1,v_2)f(v_3)-\theta_T(v_1,v_3)f(v_2)+\theta_T(v_2,v_3)f(v_1)-f([v_1,v_2,v_3]_T)\nonumber\\
   = &[Tv_1,Tv_2,f(v_3)]+[Tv_1,f(v_2),Tv_3]+[f(v_1),Tv_2,Tv_3]\nonumber\\
    &-T\Big(-D(f(v_2),Tv_1)v_3+D(f(v_1),Tv_2)v_3+\theta(Tv_2,f(v_3))v_1+\theta(f(v_2),Tv_3)v_1\nonumber\\
    &-\theta(Tv_1,f(v_3))v_2-\theta(f(v_1),Tv_3)v_2\Big)\nonumber\\
    &-f\Big(D(Tv_1,Tv_2)v_3+\theta(Tv_2,Tv_3)v_1-\theta(Tv_1,Tv_3)v_2\Big),
\end{align}
and a map $f\in C^{3}(V,L)$ is called a $3$-coboundary if there exists a map $g \in C^{1}(V, L )$ such that $f=\delta_T^{1}g$.
\item A map $f\in C^{3}(V,L)$ is called a $3$-cocycle if $\forall v_1,v_2,v_3,u_1,u_2,u_3\in V$,
\begin{align}
    &f(v_1,v_1,v_2)=0,\\
    &f(v_1,v_2,v_3)+f(v_2,v_3,v_1)+f(v_3,v_1,v_2)=0,\\
    &f(v_1,v_2,[u_1,u_2,u_3]_T)+D_T(v_1,v_2)f(u_1,u_2,u_3)\nonumber\\
    =&f([v_1,v_2,u_1]_T,u_2,u_3)+f(u_1,[v_1,v_2,u_2]_T,u_3)+f(u_1,u_2,[v_1,v_2,u_3]_T)\nonumber\\
    &+\theta_T(u_2,u_3)f(v_1,v_2,u_1)-\theta_T(u_1,u_3)f(v_1,v_2,u_2)+D_T(u_1,u_2)f(v_1,v_2,u_3).
\end{align}
\end{enumerate}
\end{defi}

 For any $\mathfrak{X}\in L \wedge L$, we  define $\partial_T(\mathfrak{X}):V \rightarrow L$ by 
 \begin{equation*}
    \partial_T(\mathfrak{X})v=TD(\mathfrak{X})v-[\mathfrak{X},Tv] ,\;\forall v\in V.
 \end{equation*}
 \begin{pro}
  Let $T$ be an $\mathcal{O}$-operator on a \textsf{L.t.sRep} pair
$(L, [\cdot,\cdot, \cdot];\theta)$. Then $\partial_T(\mathfrak{X})$ is a $1$-cocycle on the L.t.s $(V,[\c,\c,\c]_T)$ with coefficients in $(L,\theta_T)$.
 \end{pro}
 \begin{proof}
 For any $v_1,v_2,v_3\in V$, we have
 \begin{align*}
     &(\delta^{1}_T\circ \partial_T(\mathfrak{X}))(v_1,v_2,v_3)\\
     =&[Tv_1,Tv_2,TD(\mathfrak{X})v_3-[\mathfrak{X},Tv_3]]+[Tv_1,TD(\mathfrak{X})v_2-[\mathfrak{X},Tv_2],Tv_3]\\
     &+[TD(\mathfrak{X})v_1-[\mathfrak{X},Tv_1],Tv_2,Tv_3]-TD(\mathfrak{X})(D(Tv_1,Tv_2)v_3+\theta(Tv_2,Tv_3)v_1\\
     &-\theta(Tv_1,Tv_3)v_2)+[\mathfrak{X},T(D(Tv_1,Tv_2)v_3+\theta(Tv_2,Tv_3)v_1
     -\theta(Tv_1,Tv_3)v_2)]\\
     &-T\Big(D(TD(\mathfrak{X})v_1-[\mathfrak{X},Tv_1],Tv_2)v_3-D(TD(\mathfrak{X})v_2-[\mathfrak{X},Tv_2],Tv_1)v_3\Big)\\
     &-T\Big(\theta(Tv_2,TD(\mathfrak{X})v_3-[\mathfrak{X},Tv_3])v_1+\theta(TD(\mathfrak{X})v_2-[\mathfrak{X},Tv_2],Tv_3)v_1\Big)\\
     &-T\Big(-\theta(Tv_1,TD(\mathfrak{X})v_3-[\mathfrak{X},Tv_3])v_2-\theta(TD(\mathfrak{X})v_1-[\mathfrak{X},Tv_1],Tv_3)v_2\Big)\\
     \overset{\eqref{O op on lts}}{=}&-[[\mathfrak{X},Tv_1],Tv_2,Tv_3]-[Tv_1,[\mathfrak{X},Tv2],Tv_3]-[Tv_1,Tv_2,[\mathfrak{X},Tv_3]]\\
     &+[TD(\mathfrak{X})v_1,Tv_2,Tv_3]+[Tv_1,TD(\mathfrak{X})v_2,Tv_3]+[Tv_1,Tv_2,TD(\mathfrak{X})v_3]\\
     &-TD(\mathfrak{X})(D(Tv_1,Tv_2)v_3+\theta(Tv_2,Tv_3)v_1-\theta(Tv_1,Tv_3)v_2)+[\mathfrak{X},[Tv_1,Tv_2,Tv_3]]\\
     &-T\Big(D(TD(\mathfrak{X})v_1-[\mathfrak{X},Tv_1],Tv_2)v_3-D(TD(\mathfrak{X})v_2-[\mathfrak{X},Tv_2],Tv_1)v_3\Big)\\
     &-T\Big(\theta(Tv_2,TD(\mathfrak{X})v_3-[\mathfrak{X},Tv_3])v_1+\theta(TD(\mathfrak{X})v_2-[\mathfrak{X},Tv_2],Tv_3)v_1\Big)\\
     &-T\Big(-\theta(Tv_1,TD(\mathfrak{X})v_3-[\mathfrak{X},Tv_3])v_2-\theta(TD(\mathfrak{X})v_1-[\mathfrak{X},Tv_1],Tv_3)v_2\Big)\\
   \overset{\eqref{lts 2}}{=}  &[TD(\mathfrak{X})v_1,Tv_2,Tv_3]+[Tv_1,TD(\mathfrak{X})v_2,Tv_3]+[Tv_1,Tv_2,TD(\mathfrak{X})v_3]\\
     &-TD(\mathfrak{X})(D(Tv_1,Tv_2)v_3+\theta(Tv_2,Tv_3)v_1-\theta(Tv_1,Tv_3)v_2)\\
     &-T\Big(D(TD(\mathfrak{X})v_1,Tv_2)v_3-D([\mathfrak{X},Tv_1],Tv_1)v_3-D(TD(\mathfrak{X})v_2,Tv_1)v_3+D([\mathfrak{X},Tv_2],Tv_1)v_3\Big)\\
     &-T\Big(\theta(Tv_2,TD(\mathfrak{X})v_3)v_1-\theta(Tv_2,[\mathfrak{X},Tv_3])v_1+\theta(TD(\mathfrak{X})v_2,Tv_3)v_1-\theta([\mathfrak{X},Tv_2],Tv_3)v_1\Big)\\
     &-T\Big(-\theta(Tv_1,TD(\mathfrak{X})v_3)v_2+\theta(Tv_1,[\mathfrak{X},Tv_3])v_2-\theta(TD(\mathfrak{X})v_1,Tv_3)v_2+\theta([\mathfrak{X},Tv_1],Tv_3)v_2\Big)\\
    \overset{\eqref{O op on lts}}{=} &T\Big(D(Tv_1,Tv_2)D(\mathfrak{X})v_3+\theta(Tv_2,TD(\mathfrak{X})v_3)v_1-\theta(Tv_1,TD(\mathfrak{X})v_3)v_2\Big)\\
     &+T\Big(D(TD(\mathfrak{X})v_1,Tv_2)v_3+\theta(Tv_2,Tv_3)D(\mathfrak{X})v_1-\theta(TD(\mathfrak{X})v_1,Tv_3)v_2\Big)\\
     &+T\Big(D(Tv_1,TD(\mathfrak{X})v_2)v_3+\theta(TD(\mathfrak{X})v_2,Tv_3)v_1-\theta(Tv_1,Tv_3)D(\mathfrak{X})v_2\Big)\\
     &-TD(\mathfrak{X})(D(Tv_1,Tv_2)v_3+\theta(Tv_2,Tv_3)v_1-\theta(Tv_1,Tv_3)v_2)\\
     &-T\Big(D(TD(\mathfrak{X})v_1,Tv_2)v_3-D([\mathfrak{X},Tv_1],Tv_1)v_3-D(TD(\mathfrak{X})v_2,Tv_1)v_3+D([\mathfrak{X},Tv_2],Tv_1)v_3\Big)\\
     &-T\Big(\theta(Tv_2,TD(\mathfrak{X})v_3)v_1-\theta(Tv_2,[\mathfrak{X},Tv_3])v_1+\theta(TD(\mathfrak{X})v_2,Tv_3)v_1-\theta([\mathfrak{X},Tv_2],Tv_3)v_1\Big)\\
     &-T\Big(-\theta(Tv_1,TD(\mathfrak{X})v_3)v_2+\theta(Tv_1,[\mathfrak{X},Tv_3])v_2-\theta(TD(\mathfrak{X})v_1,Tv_3)v_2+\theta([\mathfrak{X},Tv_1],Tv_3)v_2\Big)\\
     =&T\Big(D(Tv_1,Tv_2)D(\mathfrak{X})v_3+\theta(Tv_2,Tv_3)D(\mathfrak{X})v_1-\theta(Tv_1,Tv_3)D(\mathfrak{X})v_2\Big)\\
     &-TD(\mathfrak{X})(D(Tv_1,Tv_2)v_3+\theta(Tv_2,Tv_3)v_1-\theta(Tv_1,Tv_3)v_2)\\
     &-T\Big(D([\mathfrak{X},Tv_2],Tv_1)v_3-D([\mathfrak{X},Tv_1],Tv_2)v_3\Big)\\
     &+T\Big(\theta(Tv_2,[\mathfrak{X},Tv_3])v_1+\theta([\mathfrak{X},Tv_2],Tv_3)v_1\Big)\\
     &-T\Big(\theta(Tv_1,[\mathfrak{X},Tv_3])v_2+\theta([\mathfrak{X},Tv_1],Tv_3)v_2\Big)\\
     =&T\Big(\theta(Tv_2,Tv_1)D(\mathfrak{X})v_3-\theta(Tv_1,Tv_2)D(\mathfrak{X})v_3+\theta(Tv_2,Tv_3)D(\mathfrak{X})v_1-\theta(Tv_1,Tv_3)D(\mathfrak{X})v_2\Big)\\
     &-TD(\mathfrak{X})\theta(Tv_2,Tv_1)v_3+TD(\mathfrak{X})\theta(Tv_1,Tv_2)v_3-TD(\mathfrak{X})\theta(Tv_2,Tv_3)v_1+TD(\mathfrak{X})\theta(Tv_1,Tv_3)v_2\\
     &-T\Big(\theta(Tv_1,[\mathfrak{X},Tv_2])v_3-\theta([\mathfrak{X},Tv_2],Tv_1)v_3-\theta(Tv_2,[\mathfrak{X},Tv_1])v_3+\theta([\mathfrak{X},Tv_1],Tv_2)v_3\Big)\\
     &+T\Big(\theta(Tv_2,[\mathfrak{X},Tv_3])v_1+\theta([\mathfrak{X},Tv_2],Tv_3)v_1\Big)\\
     &-T\Big(\theta(Tv_1,[\mathfrak{X},Tv_3])v_2+\theta([\mathfrak{X},Tv_1],Tv_3)v_2\Big)\\
     \overset{\eqref{rep lts 2}}{=}&T\Big(\theta([\mathfrak{X},Tv_1],Tv_2)v_3+\theta(Tv_1,[\mathfrak{X},Tv_2])v_3\Big)\\
     &-T\Big(\theta([\mathfrak{X},Tv_2],Tv_1)v_3+\theta(Tv_2,[\mathfrak{X},Tv_1])v_3\Big)\\
     &-T\Big(\theta([\mathfrak{X},Tv_2],Tv_3)v_1+\theta(Tv_2,[\mathfrak{X},Tv_3])v_1\Big)\\
     &+T\Big(\theta([\mathfrak{X},Tv_1],Tv_3)v_2+\theta(Tv_1,[\mathfrak{X},Tv_3])v_2\Big)\\
     &-T\Big(\theta(Tv_1,[\mathfrak{X},Tv_2])v_3-\theta([\mathfrak{X},Tv_2],Tv_1)v_3-\theta(Tv_2,[\mathfrak{X},Tv_1])v_3+\theta([\mathfrak{X},Tv_1],Tv_2)v_3\Big)\\
     &+T\Big(\theta(Tv_2,[\mathfrak{X},Tv_3])v_1+\theta([\mathfrak{X},Tv_2],Tv_3)v_1\Big)\\
     &-T\Big(\theta(Tv_1,[\mathfrak{X},Tv_3])v_2+\theta([\mathfrak{X},Tv_1],Tv_3)v_2\Big)\\
     =&0.
 \end{align*}
 Thus, we deduce that $\delta^{1}_T\circ \partial_T(\mathfrak{X})=0$. \end{proof}
 Define the set of $(2n-1)$-cochains by
 \begin{equation}
     C_{T}^{2n-1}(V,L)=\begin{cases}
     C^{2n-1}(V,L),\quad n\geq 1,\\
     L\wedge L, \quad \quad \quad\;\; \;n=0.
     \end{cases}
 \end{equation}
 Define $d_T: C_{T}^{2n-1}(V,L) \to C_{T}^{2n+1}(V,L)$ by
 \begin{equation}
     d_T=\begin{cases}
     \delta_T,\quad n\geq 1,\\
     \partial_T,\;\;\; n=0.
     \end{cases}
 \end{equation}
 We now give the cohomology of $\mathcal{O}$-operators on L.t.s.
 \begin{defi}
 Let $T$ be an $\mathcal{O}$-operator on a \textsf{L.t.sRep} pair
$(L, [\cdot,\cdot, \cdot];\theta)$.
 Denote the set of cocycles by ${Z}^{\bullet}(V,L)$, the set of coboundaries by ${B}^{\bullet}(V,L)$
 and the cohomology group by
 $${H}^{\bullet}(V,L)={Z}^{\bullet}(V,L) / {B}^{\bullet}(V,L).$$
 The cohomology group which will be taken to be the  cohomology group for the $\mathcal{O}$-operator $T$.
 \end{defi}  
 
 We need the following statement to prove the functoriality of our cohomology theory.

\begin{pro}\label{relrep}
Let $T$ and $T^{'}$ be two $\mathcal{O}$-operators on a \textsf{L.t.sRep} pair
$(L, [\cdot,\cdot, \cdot];\theta)$ and $(\phi,\psi)$ be a morphism from $T$ to $T^{'}$. Then \\
$\textbf{(i)}$ $\psi$ is a L.t.s morphism from the descendent L.t.s $(V, [\c,\c, \c]_T)$ of $T$
to the descendent L.t.s $(V, [\c,\c, \c]_{T^{'}})$ of $T^{'}$.\\
$\textbf{(ii)}$ The induced representation $(L,\theta_T)$ of the L.t.s $(V,[\c,\c,\c]_T)$ and the induced representation $(L,\theta_{T^{'}})$ satisfy the following relation
\begin{equation}
    \phi \circ \theta_T(u,v)=\theta_{T^{'}}(\psi(u),\psi(v))\circ \phi,\quad \forall u,v\in V.
\end{equation}
That is, for all $u,v\in V$, the following diagram commutes:
$$
 \xymatrix{
  L \ar[d]_{\theta_{T}(u,v)} \ar[r]^{\phi}
                & L \ar[d]^{\theta_{T^{'}}(\psi(u),\psi(v))}  \\
    L \ar[r]_{\phi}
                & L            .}
$$
\end{pro}
\begin{proof}
By, $\eqref{c1}$, $\eqref{c2}$, we have 
\begin{align*}
    &\psi([u,v,w]_T)\;=\psi\Big(D(T(u),T(v))w-\theta(T(u),T(w))v+\theta(T(v),T(w))u\Big)\\
    &\quad \quad \quad\quad\;\;\; \;\;\;=D(\phi T(u),\phi T(v))\psi(w)-\theta(\phi T(u),\phi T(w))\psi(v)+\theta(\phi T(v),\phi T(w))\psi(u)\\
    &\quad \quad \quad\quad\;\;\; \;\;\;=D( T^{'} \psi(u), T^{'}\psi(v))\psi(w)-\theta( T^{'}\psi(u), T^{'}\psi(w))\psi(v)+\theta( T^{'}\psi(v), T^{'}\psi(w))\psi(u)\\
    &\quad \quad \quad\quad\;\;\; \;\;\;=[\psi(u),\psi(v),\psi(w)]_{T^{'}}.
\end{align*}
Now, according to $\eqref{c1}$, $\eqref{c2}$ and $\eqref{indurep}$, for all $u,v\in V,\; x\in L$, we have 
\begin{align*}
    &\phi(\theta_{T}(u,v)x)\;=\phi([x,T(u),T(v)])+\phi(T(\theta(x,T(v))u))-\phi(T(D(x,T(u))v))\\
    &\quad \quad \quad\quad\;\;\; \;\;\;=[\phi(x),\phi(T(u)),\phi(T(v))]+T^{'}(\psi(\theta(x,T(v))u))-T^{'}(\psi(D(x,T(u))v))\\
    &\quad \quad \quad\quad\;\;\; \;\;\;=[\phi(x),T^{'}(\psi(u)),T^{'}(\psi(v))]+T^{'}\Big(\theta(\phi(x),T^{'}(\psi(v)))\psi(u)-D(\phi(x),T^{'}(\psi(u)))\psi(v)\Big)\\
     &\quad \quad \quad\quad\;\;\; \;\;\;=\theta_{T^{'}}(\psi(u),\psi(v))\phi(x).
\end{align*}
\end{proof}
Let $T$ and $T^{'}$ be two $\mathcal{O}$-operators on $L$ with respect to the representation $(V,\theta)$ and  $(\phi,\psi)$ a morphism from $T$ to $T^{'}$ with $\psi$ invertible. Denote by  $C^{2n-1}_T(V_T,L)$ the space of $(2n-1)$-cochains of L.t.s $(V,[\c,\c,\c]_T)$  with coefficients on the representation $(L,\theta_T)$. Define
\begin{equation*}
  \begin{cases}
  \gamma:C^{2n-1}_{T}(V_T,L)\rightarrow C^{2n-1}_{T^{'}}(V_{T^{'}},L),\quad \quad n\geq 1,\\
  \gamma:L\wedge L\rightarrow L\wedge L,\quad \quad \quad \quad \quad \quad\quad \quad \;\;\; n=0.
  \end{cases}  
\end{equation*} by
\begin{equation*}
\begin{cases} \gamma(f)(u_1,\cdots,u_{2n-1})=\phi(f(\psi^{-1}(u_1),\cdots,\psi^{-1}(u_{2n-1}))),\; \forall u_i\in V,\quad \quad n\geq 1,\\
\gamma(\mathfrak{X})=\phi(\mathfrak{X})=(\phi(x_1),\phi(x_2)),\quad \quad\quad  \quad \quad\forall\mathfrak{X}=(x_1,x_2)\in L\wedge L,\quad\;\;\;  n=0.
\end{cases}
\end{equation*}
\begin{thm}
With the above notations, $\gamma$ is a cochain map from the cochain complex $(C^{*}_{T}(V_T,L),d_{T})$ to the cochain complex $(C^{*}_{T^{'}}(V_{T^{'}},L),d_{T^{'}})$. Consequently, it induces a morphism $\overline{\gamma}$ from the cohomology group $H^{2n-1}_{T}(V_T,L)$ to $H^{2n-1}_{T^{'}}(V_{T^{'}},L)$, for all $n\geq 0$.
\end{thm}
\begin{proof}
For $n=0$, let $\mathfrak{X}\in L\wedge L$ and $(\phi,\psi)$ be a morphism from $T$ to $T^{'}$.  Then we have 
\begin{align*}
    &(\partial_{T^{'}}\gamma(\mathfrak{X}))(v)=T^{'}(D(\gamma(\mathfrak{X}))v)-[\gamma(\mathfrak{X}),T^{'}v]\\
    &\quad \quad \quad \quad\;\; \;\;\;\; =T^{'}(D(\phi(\mathfrak{X}))\psi\circ \psi^{-1}(v))-[\phi(\mathfrak{X}),T^{'}\psi\circ\psi^{-1}(v)]\\
    &\quad \quad\;\;\; \;\;\;\;\;\overset{\eqref{c1}+\eqref{c2}}{=}T^{'}\psi(D(\mathfrak{X})\psi^{-1}(v))-[\phi(\mathfrak{X}),\phi(T\psi^{-1}(v))]\\
    &\quad \quad \;\;\; \;\;\; \;\;\;\;\;\;\overset{\eqref{c1}}{=}\phi T(D(\mathfrak{X})\psi^{-1}(v))-\phi([\mathfrak{X},T\psi^{-1}(v)])\\
    &\quad \quad \quad \quad\; \;\;\;\;\;=\phi(\partial_T(\mathfrak{X}))(\psi^{-1}(v))=\gamma(\partial_T(\mathfrak{X})(v)).
\end{align*}
Now, for any $f\in C^{2n-1}_{T}(V_T,L)(n\geq 1) $, by Proposition \ref{relrep}, we have
{\small\begin{align*}
  &(\delta_{T^{'}} \gamma(f))(u_1,u_2, \cdots , u_{2n+1})\nonumber\\
    =&\theta_{T^{'}}(u_{2n},u_{2n+1})\gamma(f)(u_1,u_2, \cdots , u_{2n-1})- \theta_{T^{'}}(u_{2n-1},u_{2n+1})\gamma(f)(u_1,u_2, \cdots , u_{2n-2},u_{2n}) \nonumber \\
    & +\sum_{k=1}^n (-1)^{n+k}D_{T^{'}}(u_{2k-1},u_{2k})\gamma(f)(u_1,u_2, \cdots , \widehat{u}_{2k-1},\widehat{u}_{2k}, \cdots , u_{2n+1}) \nonumber \\
    &+ \sum_{k=1}^n \sum_{j=2k+1}^{2n+1}  (-1)^{n+k+1} \gamma(f)(u_1,u_2, \cdots,  \widehat{u}_{2k-1},\widehat{u}_{2k}, \cdots, [u_{2k-1},u_{2k},u_j]_{T^{'}},\cdots , u_{2n+1}) \\
    =&\theta_{T^{'}}(u_{2n},u_{2n+1})\phi f\Big(\psi^{-1}(u_1),\psi^{-1}(u_2), \cdots , \psi^{-1}(u_{2n-1})\Big)\\
    &- \theta_{T^{'}}(u_{2n-1},u_{2n+1})\phi f\Big(\psi^{-1}(u_1),\psi^{-1}(u_2), \cdots , \psi^{-1}(u_{2n-2}),\psi^{-1}(u_{2n})\Big) \nonumber \\
    & +\sum_{k=1}^n (-1)^{n+k}D_{T^{'}}(u_{2k-1},u_{2k})\phi f\Big(\psi^{-1}(u_1),\psi^{-1}(u_2), \cdots , \widehat{\psi^{-1}(u_{2k-1})},\widehat{\psi^{-1}(u_{2k})}, \cdots , \psi^{-1}(u_{2n+1})\Big) \nonumber \\
    &+ \sum_{k=1}^n \sum_{j=2k+1}^{2n+1}  (-1)^{n+k+1}
    \phi f\Big(\psi^{-1}(u_1),\psi^{-1}(u_2), \cdots,  \widehat{\psi^{-1}(u_{2k-1})},\widehat{\psi^{-1}(u_{2k})},\\
    &\cdots, [\psi^{-1}(u_{2k-1}),\psi^{-1}(u_{2k}),\psi^{-1}(u_j)]_{T},\cdots , \psi^{-1}(u_{2n+1})\Big) \\
    =&\theta_{T^{'}}(\psi \circ \psi^{-1}(u_{2n}),\psi \circ \psi^{-1}(u_{2n+1}))\phi f\Big(\psi^{-1}(u_1),\psi^{-1}(u_2), \cdots , \psi^{-1}(u_{2n-1})\Big)\\
    &- \theta_{T^{'}}(\psi \circ \psi^{-1}(u_{2n-1}),\psi \circ \psi^{-1}(u_{2n+1}))\phi f\Big(\psi^{-1}(u_1),\psi^{-1}(u_2), \cdots , \psi^{-1}(u_{2n-2}),\psi^{-1}(u_{2n})\Big) \nonumber \\
    & +\sum_{k=1}^n (-1)^{n+k}D_{T^{'}}(\psi \circ \psi^{-1}(u_{2k-1}),\psi \circ \psi^{-1}(u_{2k}))\phi f\Big(\psi^{-1}(u_1),\psi^{-1}(u_2), \cdots , \widehat{\psi^{-1}(u_{2k-1})},\widehat{\psi^{-1}(u_{2k})},\\
    &\cdots , \psi^{-1}(u_{2n+1})\Big) + \sum_{k=1}^n \sum_{j=2k+1}^{2n+1}  (-1)^{n+k+1}
    \phi f\Big(\psi^{-1}(u_1),\psi^{-1}(u_2), \cdots,  \widehat{\psi^{-1}(u_{2k-1})},\widehat{\psi^{-1}(u_{2k})},\\
    &\cdots, [\psi^{-1}(u_{2k-1}),\psi^{-1}(u_{2k}),\psi^{-1}(u_j)]_{T},\cdots , \psi^{-1}(u_{2n+1})\Big)\\
    =&\phi\Big(\theta_{T}( \psi^{-1}(u_{2n}), \psi^{-1}(u_{2n+1})) f\Big(\psi^{-1}(u_1),\psi^{-1}(u_2), \cdots , \psi^{-1}(u_{2n-1})\Big)\\
    &- \theta_{T}( \psi^{-1}(u_{2n-1}),\psi^{-1}(u_{2n+1})) f\Big(\psi^{-1}(u_1),\psi^{-1}(u_2), \cdots , \psi^{-1}(u_{2n-2}),\psi^{-1}(u_{2n})\Big) \nonumber \\
    & +\sum_{k=1}^n (-1)^{n+k}D_{T}(\psi^{-1}(u_{2k-1}), \psi^{-1}(u_{2k})) f\Big(\psi^{-1}(u_1),\psi^{-1}(u_2), \cdots , \widehat{\psi^{-1}(u_{2k-1})},\widehat{\psi^{-1}(u_{2k})},\\
    &\cdots , \psi^{-1}(u_{2n+1})\Big)  +\sum_{k=1}^n \sum_{j=2k+1}^{2n+1}  (-1)^{n+k+1}
     f\Big(\psi^{-1}(u_1),\psi^{-1}(u_2), \cdots,  \widehat{\psi^{-1}(u_{2k-1})},\widehat{\psi^{-1}(u_{2k})},\\
    &\cdots, [\psi^{-1}(u_{2k-1}),\psi^{-1}(u_{2k}),\psi^{-1}(u_j)]_{T},\cdots , \psi^{-1}(u_{2n+1})\Big) \Big)\\
    =&\phi \Big(\delta_{T}(f)(\psi^{-1}(u_1),\cdots,\psi^{-1}(u_{2n+1}))\Big)=\gamma(\delta_{T}(f))(u_1,\cdots,u_{2n+1}).
\end{align*}}
Thus $\gamma$ is a cochain map. Consequently, it induces a morphism $\overline{\gamma}$ from the cohomology group $H^{2n-1}_{T}(V_T,L)$ to $H^{2n-1}_{T^{'}}(V_{T^{'}},L)$, for all $n\geq 0$.
\end{proof}
\section{Deformation of an $\mathcal{O}$-operator on a Lie triple system}
In this section, we will apply the classical deformation theory due to  Gerstenhaber (infinitesimal and formal deformations) to an $\mathcal{O}$-operator on a L.t.s. We will introduce the notion of 
Nijenhuis element associated to an $\mathcal{O}$-operator that arise from trivial deformations.
We also consider the rigidity of an $\mathcal{O}$-operator and provide a sufficient condition  in terms of Nijenhuis elements. 
\subsection{Infinitesimal deformations}
Let 
$(L, [\cdot,\cdot, \cdot],\theta)$ be a \textsf{L.t.sRep} pair. Suppose  $T : V \rightarrow L$ is an $\mathcal{O}$-operator associated to representation $(V, \theta)$.
\begin{defi}
A  parametrized sum $T_t = T + tT_1$,  for some $T_1 \in Hom(V, L)$, is called an infinitesimal deformation of $T$ if $T_t$ is an  $\mathcal{O}$-operator for all values of $t$. In this case, we say that
$T_1$ generates an infinitesimal deformation of $T$.
\end{defi}
Suppose $T_1$ generates an infinitesimal deformation of $T$. Then we have
\begin{equation*}
    [T_t(u),T_t(v),T_t(w)]=T_t\Big(D(T_t(u),T_t(v))w+\theta(T_t(v),T_t(w))u-\theta(T_t(u),T_t(w))v\Big), \quad \forall u,v,w \in V.
\end{equation*}
This is equivalent to the following conditions
\begin{align}\label{eq:defo}
&[Tu,Tv,T_1(w)]+[Tu,T_1(v),Tw]+[T_1(u),Tv,Tw]\nonumber\\
=&T\Big(D(Tu,T_1(v))w+D(T_1(u),Tv)w+\theta(Tv,T_1(w))u+\theta(T_1(v),Tw)u\nonumber\\
&-\theta(Tu,T_1(w))v-\theta(T_1(u),Tw)v\Big)+T_1\Big(D(Tu,Tv)w+\theta(Tv,Tw)u-\theta(Tu,Tw)v\Big),
\end{align}
\begin{align}
&[Tu,T_1(v),T_1(w)]+[T_1(u),Tv,T_1(w)]+[T_1(u),T_1(v),Tw]\nonumber\\
=&T\Big(D(T_1(u),T_1(v))w+\theta(T_1(v),T_1(w))u-\theta(T_1(u),T_1(w))v\Big)\nonumber\\
&+T_1\Big(D(Tu,T_1(v))w+D(T_1(u),Tv)w+\theta(Tv,T_1(w))u\nonumber\\
&+\theta(T_1(v),Tw)u-\theta(Tu,T_1(w))v-\theta(T_1(u),Tw)v\Big),
\end{align}
and
\begin{align}\label{defo1}
&[T_1(u),T_1(v),T_1(w)]=T_1\Big(D(T_1(u),T_1(v))w+\theta(T_1(v),T_1(w))u-\theta(T_1(u),T_1(w))v\Big).
\end{align}
Note that, by Eq.\eqref{lts 1}, we have 
\begin{align*}
    &[Tu,Tv,T_1(w)]=-[Tv,T_1(w),Tu]-[T_1(w),Tu,Tv]= [T_1(w),Tv,Tu]+[Tu,T_1(w),Tv].
\end{align*}
Then the identity \eqref{eq:defo} implies that $T_1$ is a $1$-cocycle (closed) in the cohomology of $T$. Hence, $T_1$ defines a
cohomology class in $\mathcal{H}^1(V,L)$. Eq. \eqref{defo1} means that $T_1$ is an $\mathcal{O}$-operator on the \textsf{L.t.sRep} pair
$(L, [\cdot,\cdot, \cdot];\theta)$.

\begin{defi}Let  $T_t = T + tT_1$ and $T^{'}_t = T + tT^{'}_1$ be
two infinitesimal  deformations of an  $\mathcal{O}$-operator $T$. We say that  $T_t $ and $T^{'}_t $ are equivalent if there exists an element $\mathfrak{X}\in L\wedge L$ such that the pair
\begin{equation}
\Big(\phi_{t}=Id_L+t\mathcal{L}(\mathfrak{X})-=Id_L+t[\mathfrak{X},-],\;\psi_{t}=Id_V+tD(\mathfrak{X})(-)\Big),
\end{equation}
defines a morphism of  $\mathcal{O}$-operators from  $T_t$ to $T^{'}_{t}$.
\end{defi}
Since  $\phi_{t}=Id_L+t[\mathfrak{X},-]$ is a L.t.s morphism of $(L, [\cdot,\cdot, \cdot])$, then we have the following conditions
\begin{equation}\label{condition1}
\begin{cases}
[z_1,[\mathfrak{X},z_2],[\mathfrak{X},z_3]]+[[\mathfrak{X},z_1],z_2,[\mathfrak{X},z_3]]
+[[\mathfrak{X},z_1],[\mathfrak{X},z_2],z_3]=0,\\
[[\mathfrak{X},z_1],[\mathfrak{X},z_2],[\mathfrak{X},z_3]]=0,\;\text{for}\;z_1,z_2,z_3\in L.
\end{cases}
\end{equation}
The condition $\psi_t(\theta(z_1,z_2)u)=\theta(\phi_t(z_1),\phi_t(z_2))\psi_t(u)$ gives that
\begin{equation}\label{condition2}
\begin{cases}
\theta(z_1,[\mathfrak{X},z_2])D(\mathfrak{X})+\theta([\mathfrak{X},z_1],z_2)D(\mathfrak{X})+\theta([\mathfrak{X},z_1],[\mathfrak{X},z_2])=0,\\
\theta([\mathfrak{X},z_1],[\mathfrak{X},z_2])D(\mathfrak{X})=0.
\end{cases}
\end{equation}
Finally, the condition  $\phi_t\circ T_t=T^{'}_{t} \circ \psi_t$
is equivalent to

\begin{equation}\label{eq:defoo}
\begin{cases}
T_1(u)+[\mathfrak{X},Tu]=TD(\mathfrak{X})u+T^{'}_1(u),\\
[\mathfrak{X},T_1(u)]=T^{'}_1D(\mathfrak{X})(u).
\end{cases}
\end{equation}
Note that the above identities hold for all $\mathfrak{X}\in L\wedge L$, $z_1,z_2,z_3\in L$ and $u\in V$.\\

From the first condition of \eqref{eq:defoo}, we have
\begin{align*}
&T_1(u)-T^{'}_1(u)=TD(\mathfrak{X})u-[\mathfrak{X},Tu]=(d_T\mathfrak{X})(u).
\end{align*}
Therefore, we get the following result.
\begin{thm}
Let $T_t = T + tT_1$ and $T^{'}_t = T + tT^{'}_1$ be two equivalent infinitesimal deformations of an
$\mathcal{O}$-operator $T$. Then $T_1$ and $T^{'}_1$ defines the same cohomology class in ${H}_{T}^{1}(V;L)$.
\end{thm}
\begin{defi}
Let $T_t = T + tT_1$ be an infinitesimal deformation  of an $\mathcal{O}$-operator $T$. The deformation  $T_t $ is said to be
trivial if $(\phi_{t},\psi_{t})$ is a morphism
from $T_t$ to $T$.
\end{defi}
Now we will define Nijenhuis elements associated to an $\mathcal{O}$operator on a L.t.s.
\begin{defi}
Let $(L, [\cdot,\cdot, \cdot];\theta)$ be a \textsf{L.t.sRep} pair and  $T$ be an $\mathcal{O}$-operator. An element $\mathfrak{X}\in L\wedge L$ is called a Nijenhuis element associated to $T$ if $\mathfrak{X}$ satisfies Eqs.\eqref{condition1}, \eqref{condition2} and the equation 
\begin{equation}\label{condition3}
    [\mathfrak{X},TD(\mathfrak{X})(u)-[\mathfrak{X},T(u)]]=0.
\end{equation}
The set of Nijenhuis elements associated to an $\mathcal{O}$-operator $T$ is denoted by  Nij$(T )$.
\end{defi}
Our motivation to introduce the above definition is that a trivial infinitesimal deformation gives rise to a Nijenhuis
element. We give in the next subsection a sufficient
condition to the rigidity of an $\mathcal{O}$-operator in terms of Nijenhuis elements.

\subsection{Formal deformations}
In this subsection, we consider formal deformations of $\mathcal{O}$-operators generalizing the
classical deformation theory of Gerstenhaber \cite{Gerstenhaber}. Let $(L, [\cdot,\cdot, \cdot];\theta)$ be a \textsf{L.t.sRep} pair and $T:V\rightarrow L$  be an $\mathcal{O}$-operator.

Let $\mathbb{K}[[t]]$ be the ring of power series in one variable $t$. For any $\mathbb{K}$-linear space $L$, we denote by $L[[t]]$ the
vector space of formal power series in $t$ with coefficients in $V$. If in addition, we have a structure of L.t.s $(L, [\cdot,\cdot, \cdot])$ over $\mathbb{K}$, then there is a L.t.s structure over the ring $\mathbb{K}[[t]]$ on $L[[t]]$ given
by
\begin{equation}\label{eq:power3-lie}
\Big[\sum_{i=0}^{+\infty}x_it^{i},\sum_{j=0}^{+\infty}y_jt^{j},\sum_{k=0}^{+\infty}z_kt^{k}\Big]=\sum_{s=0}^{+\infty}\sum_{i+j+k=s}[x_i,y_j,z_k]t^{s},\;\forall x_i,y_j,z_k\in L.
\end{equation}
For any representation $(V,\theta)$ of $(L, [\cdot,\cdot, \cdot])$, there is a natural representation of the L.t.s
$L[[t]]$ on the $\mathbb{K}[[t]]$-module $V[[t]]$, which is given by
\begin{equation}
\theta\Big(\sum_{i=0}^{+\infty}x_it^{i},\sum_{j=0}^{+\infty}y_jt^{j}\Big)\Big(\sum_{k=0}^{+\infty}V_kt^{k}\Big)=\sum_{s=0}^{+\infty}\sum_{i+j+k=s}\theta(x_i,y_j)v_kt^{s},\;\forall x_i,y_j\in L,\;v_k\in V.
\end{equation}
 Consider a power series
\begin{equation}\label{eq:formal1}
T_t=\sum_{i=0}^{+\infty}T_it^{i},\;T_i\in Hom_{\mathbb{K}}(V;L),
\end{equation}
that is, $T_t\in Hom_{\mathbb{K}}(V;L)[[t]]=Hom_{\mathbb{K}}(V;L[[t]])$. Extend it to be a $\mathbb{K}[[t]]$-module map from
$V[[t]]$ to $L[[t]]$ which is still denoted by $T_t$.
\begin{defi}
If $T_t=\displaystyle\sum_{i=0}^{+\infty}T_it^{i}$ with $T_0=T$ satisfies
\begin{align}\label{eq:formal2}
&[T_t(u),T_t(v),T_t(w)]=T_t\Big(D(T_t(u),T_t(v))w+\theta(T_t(v),T_t(w))u-\theta(T_t(u),T_t(w))v\Big),
\end{align}
we say that $T_t$ is a formal deformation of the  $\mathcal{O}$-operator $T$.
\end{defi}
Recall that a formal deformation of a L.t.s $(L, [\cdot,\cdot, \cdot])$  is a formal power series $ \omega_t=\displaystyle\sum_{k=0}^{+\infty}\omega_kt^{k}$
where $\omega_k\in Hom((\wedge^{2}L)\otimes L; L)$ such that $\omega_0(x, y, z)=[x, y, z]$ for any $x, y, z \in L$ and $\omega_t$ defines a L.t.s structure over the ring $\mathbb{K}[[t]]$ on $L[[t]]$.\\
Based on the relationship between  $\mathcal{O}$-operators and L.t.s structure, we have
\begin{pro}
Let  $T_t=\displaystyle\sum_{i=0}^{+\infty}T_it^{i}$ be a formal deformation of an  $\mathcal{O}$-operator $T$ on the \textsf{L.t.sRep} pair
$(L, [\cdot,\cdot, \cdot];\theta)$. Then $[\cdot,\cdot,\cdot]_{T_t}$ defined by
\begin{align*}
\quad\;[u,v,w]_{T_t}
&=\displaystyle\sum_{k=0}^{+\infty}\sum_{i+j=k}\Big(D(T_i(u),T_j(v))w+\theta(T_i(v),T_j(w))u-\theta(T_i(u),T_j(w))v\Big)t^{k},
\end{align*} for all $u,v,w\in V$,
is a formal deformation of the associated L.t.s $(V,[\cdot,\cdot,\cdot]_T)$ given in Lemma \ref{struV}.
\end{pro}
By applying Eqs.\eqref{eq:power3-lie}-\eqref{eq:formal1} to expand Eq.\eqref{eq:formal2} and collecting coefficients of $t^{s}$,
we see that Eq.\eqref{eq:formal2} is equivalent to the following system of equations
\begin{align}\label{eq:system}
&\sum_{\overset{{i+j+k=s}}{i,j,k\geq 0}}[T_i(u),T_j(v),T_k(w)]\nonumber\\
&=\sum_{\overset{{i+j+k=s}}{i,j,k\geq 0}}T_i\Big(D(T_j(u),T_k(v))w+\theta(T_j(v),T_k(w))u-\theta(T_j(u),T_k(w))v\Big),
\end{align}for all $s\geq 0,\;u,v,w\in V.$ 
\begin{pro}\label{pro1}
Let $T_t=\displaystyle\sum_{i=0}^{+\infty}T_it^{i}$ be a formal deformation of an $\mathcal{O}$-operator on the  \textsf{L.t.sRep} pair $(L, [\cdot,\cdot, \cdot];\theta)$. Then  $T_1$ is a $1$-cocycle in the cohomology of an $\mathcal{O}$-operator $T$, that is, $d_T (T_1)=0$.
\end{pro}
\begin{proof}
Note that \eqref{eq:system} holds for $s=0$ as $T_0=T$ is an $\mathcal{O}$-operator. For $s=1$, we get
\begin{align*}
&[Tu,Tv,T_1(w)]+[Tu,T_1(v),Tw]+[T_1(u),Tv,Tw]\\
&=T\Big(D(Tu,T_1(v))w+D(T_1(u),Tv)w+\theta(Tv,T_1(w))u+\theta(T_1(v),Tw)u\\
&-\theta(Tu,T_1(w))v-\theta(T_1(u),Tw)v\Big)+T_1\Big(D(Tu,Tv)w+\theta(Tv,Tw)u-\theta(Tu,Tw)v\Big),
\end{align*}
 for all $u,v,w\in V$. This implies that $(d_T (T_1))(u, v,w) = 0$. Hence the linear term $T_1$ is a $1$-cocycle in the cohomology of $T$.
 \end {proof}
 \begin{defi}
Let $T$ be an $\mathcal{O}$-operator on the  \textsf{L.t.sRep} pair $(L, [\cdot,\cdot, \cdot];\theta)$. The $1$-cocycle given in Proposition \ref{pro1} is called
the infinitesimal of the formal deformation $T_t=\displaystyle\sum_{i=0}^{+\infty}T_it^{i}$ of $T$.
 \end{defi}

In the sequel, we discuss equivalent formal deformations.
\begin{defi}
Let $T_t=\displaystyle\sum_{i=0}^{+\infty}T_it^{i}$ and $T^{'}_t=\displaystyle\sum_{i=0}^{+\infty}T{'}_it^{i}$ be two formal deformations of an $\mathcal{O}$-operator $T=T_0=T^{'}_0$ on a \textsf{L.t.sRep} pair
$(L, [\cdot,\cdot, \cdot];\theta)$. They are said to be equivalent if there exists
an element $\mathfrak{X}\in L\wedge L$, $\phi_i\in gl(L)$ and $\psi_i\in gl(V)$, $i\geq2$, such that the pair
\begin{equation}\label{eq:equi}
\Big(\phi_t=Id_L+t[\mathfrak{X},-]+\sum_{i=2}^{+\infty}\phi_it^{i},\;\psi_t=Id_V+tD(\mathfrak{X})(-)+\sum_{i=2}^{+\infty}\psi_it^{i}\Big),
\end{equation}
is a morphism of  $\mathcal{O}$-operators from  $T_t$ to $T^{'}_t$.
In particular, a formal deformation $T_t$ of an $\mathcal{O}$-operator $T$ is said to be trivial if there
exists an $\mathfrak{X}\in L\wedge L$, $\phi_i \in  gl(L)$ and $\psi_i \in gl(V)$, $i \geq 2$, such that $(\phi_t, \psi_t )$ defined by Eq. \eqref{eq:equi}
gives an equivalence between $T_t$ and $T$, with the latter regarded as a deformation of
itself.
\end{defi}
\begin{thm}
If two formal deformations of an $\mathcal{O}$-operator $T$ on a \textsf{L.t.sRep} pair
$(L, [\cdot,\cdot, \cdot];\theta)$ are equivalent, then their infinitesimals
are in the same cohomology class in ${H}_{T}^{1}(V;L)$.
\end{thm}
\begin{proof} Let $(\phi_t,\psi_t)$ be the two maps defined by Eq.\eqref{eq:equi} which gives an equivalence between
two deformations $T_t=\displaystyle\sum_{i=0}^{+\infty}T_it^{i}$ and $T^{'}_t=\displaystyle\sum_{i=0}^{+\infty}T{'}_it^{i}$ of an $\mathcal{O}$-operator $T$. By
$(\phi_t \circ T_t)(u)=(T^{'}_t\circ \psi_t)(u)$, we have
\begin{align*}
&T_1(u)=T^{'}_1(u)+TD(\mathfrak{X})u-[\mathfrak{X},Tu]_\mathfrak{g}\\
&\quad \;\;\;\;\; =T^{'}_1(u)+(d_T\mathfrak{X})(u),\;\forall u\in V,
\end{align*}
which implies that $T_1$ and $T^{'}_1$ are in the same cohomology class.\end{proof}
\begin{defi}
An $\mathcal{O}$-operator $T$ is said to be rigid if any
formal deformation of $T_t$ is trivial.
\end{defi}

Now, we give  a cohomological characterization  of  rigidity of an $\mathcal{O}$-operator involving the set of Nijenhuis elements.

\begin{thm}
Let $T$ be an $\mathcal{O}$-operator. If $Z_{T}^{1}(V,L)=d_T(Nij(T))$, then $T$ is rigid.
\begin{proof}
Let $T_t =\sum_{i=0}^{\infty} t^iT_i$ be any formal deformation of $T$. We have seen in Proposition \ref{pro1} that $T_1$ is a $1$-cocycle in the
cohomology of $T$, i.e. $T_1 \in Z^{1}_{T}
(V, L)$. Thus, from the hypothesis, we get a Nijenhuis element $\mathfrak{X} \in Nij(T )$
such that $T_1 = d_T (\mathfrak{X})$. Then, setting
\begin{equation*}
    \phi_t=Id_L+t[\mathfrak{X},-],\quad \quad \psi_t=Id_V+tD(\mathfrak{X})(-),
\end{equation*}
and define $T^{'}_t=\phi_t\circ T_t \circ \psi_{t}^{-1}$, one obtains  $T^{'}_t$
 is a deformation equivalent to $T_t$. Moreover, we have
\begin{align*}
     &T^{'}_t(u)=(Id_L+t[\mathfrak{X},-]) (T_t \Big(Id_V-tD(\mathfrak{X})+t^{2}D^{2}(\mathfrak{X})+\cdots+(-1)^{i}t^{i}D^{i}(\mathfrak{X})+\cdots)(u)\Big))\\
     &\quad\;\;\;\;\;=T(u)+t(T_1(u)-TD(\mathfrak{X})(u)+[\mathfrak{X},T(u)])+t^{2}T^{'}_{2}(u)+\cdots\\
       &\quad\;\;\;\;\; =T(u)+t^{2}T^{'}_{2}(u)+\cdots\quad (as\; T_1(u)=d_T(\mathfrak{X})(u)).
\end{align*}
 Hence the coefficient of $t$ in the expression of $T^{'}_
t$ is trivial. By applying the same process repeatedly, we get that $T_t$ is equivalent to $T$. Therefore, 
$T$ is rigid.
\end{proof}
\end{thm}
\section{From cohomology groups of $\mathcal{O}$-operators on Lie algebras to those on Lie triple systems}

Motivated by the construction of L.t.s from  Lie algebras. We give some connections between $\mathcal{O}$-operators on Lie algebras and L.t.s.\\

A representation of a Lie algebra $(L,[\c,\c])$ on a vector space $V$ is a linear map $\rho: L \rightarrow End(V)$  such that 
\begin{equation*}
    \rho([x,y])=\rho(x)\rho(y)-\rho(y)\rho(x),\;\forall x,y\in L.
\end{equation*}
Similarly, we call the pair of Lie algebra $(L, [\c, \c])$ and
the representation $\rho$ a \textsf{LieRep} pair, denoted by $(L,[\c,\c];\rho)$.
Let $(L,[\c,\c],\rho)$ be a \textsf{LieRep} pair.  Define a bilinear bracket $[\c, \c]_{\rho}$ on $L\oplus V$ by
\begin{equation*}
  [x+u, y+v]_{\rho}=[x,y]+\rho(x)v-\rho(y)u,\;\forall x,y\in L,\; u,v\in V.  
\end{equation*}
Then $(L\oplus V,[\c, \c]_{\rho})$ is a  semi-direct product Lie algebra, denoted by $L\ltimes_{\rho}V$.\\

Let $(L,[\c,\c];\rho)$ be a \textsf{LieRep} pair. The space of $p$-cochains is $C_{Lie}^{p}(L,V)=Hom(\wedge^{p}L,V)$ for $p\geq 0$.  The
coboundary operator $\partial_{\rho}:C_{Lie}^{p}(L,V)\rightarrow C_{Lie}^{p+1}(L,V)$ is defined by
\begin{align*}
    &\partial_{\rho}(f)(x_1,\cdots,x_{p+1})\\
    &=\sum_{1\leq i < j \leq p+1}(-1)^{i+j}f([x_i,x_j],x_1,\cdots,\widehat{x_i},\cdots,\widehat{x_j},\cdots,x_{p+1})\\
    &+\sum_{i=1}^{p+1}(-1)^{i+1}\rho(x_i)f(x_1,\cdots,\widehat{x_i},\cdots,x_{p+1}),
\end{align*}
for $f\in C_{Lie}^{p}(L,V)(p \geq 0) $. We denote the corresponding $p$-th cohomology group by $H_{Lie}^{p}(L,V)$. 
\begin{defi}
A linear map $T:V \rightarrow L$ is called an $\mathcal{O}$-operator on a \textsf{LieRep} pair $(L,[\c,\c];\rho)$ if $T$ satisfies 
\begin{equation}
    [Tu,Tv]=T\Big(\rho(Tu)v-\rho(Tv)u\Big),\;\forall u,v\in V.
\end{equation}
\end{defi}
Let $T$ be an $\mathcal{O}$-operator on a LieRep pair $(L,[\c,\c];\rho)$. There is a Lie algebra structure on $(V,[\c,\c]_T)$, where the bracket $[\c,\c]_T:V\times V \rightarrow V$ is given by
\begin{equation}
    [u,v]_T=\rho(Tu)v-\rho(Tv)u\;\forall u,v\in V.
\end{equation}
Now, we recall some results from \cite{Tang}. There is a representation of $(V,[\c,\c]_T)$, $\rho_T:V\rightarrow gl(L)$, defined by
\begin{equation*}
\rho_T(u)x=[Tu,x]+T\rho(x)u,\; \forall u\in V,x\in L.    
\end{equation*}
Let $T$ be an $\mathcal{O}$-operator on a LieRep pair $(L,[\c,\c];\rho)$. Consider the set of $p$-cochains $C_{T}^{p}(V,L)=Hom(\wedge^{p}V,L)$. Let $\widetilde{d}_T:C_{T}^{p}(V,L)\rightarrow C_{T}^{p+1}(V,L) (p \geq 0)$ be the corresponding coboundary operator of the
Lie algebra $(V, [\c, \c]_T )$ with coefficients in the representation $(L,\rho_T)$, defined  for all $f\in C_{T}^{p}(V,L)$ and $u_1,\cdots,u_{p+1}\in V$ by

\begin{align*}
&\widetilde{d}_T(f)(u_1,\cdots,u_{p+1})\\
&=\sum_{1\leq i < j \leq p+1}(-1)^{i+j}f(\rho(Tu_i)u_j-\rho(Tu_j)u_i,u_1,\cdots,\widehat{u_i},\cdots,\widehat{u_j},\cdots,u_{p+1})\\
&+\sum_{i=1}^{p+1}(-1)^{i+1}\Big([Tu_i,f(u_1,\cdots,\widehat{u_i},\cdots,u_{p+1})]+T\rho(f(u_1,\cdots,\widehat{u_i},\cdots,u_{p+1}))(u_i) \Big).
\end{align*}

Then the cochain complex
$(C_{T}^{p}(V,L),)$ is called a cochain complex of the  $\mathcal{O}$-operator $T$ on a LieRep pair $(L,[\c,\c];\rho)$.
We denote the corresponding $p$-th cohomology group by $H_{T}^{p}(V,L)$.\\
As we know, there is a method of constructing L.t.s from Lie algebras given in Example \ref{lts}.
\begin{thm}
Let $(L,[\c,\c];\rho)$ be a LieRep pair. Define $\theta_{\rho}:\otimes^{2} L \rightarrow gl(V)$ by
\begin{equation}
    \theta_{\rho}(x,y)=\rho(y)\rho(x).
\end{equation}
Then $(L, [\c,\c,\c]=[\c,\c]\circ ([\c,\c]\otimes Id_L);\theta_{\rho})$ is a \textsf{L.t.sRep} pair. 
\end{thm}
\begin{proof}Let $(L,[\c,\c];\rho)$ be a LieRep pair. We know that there is a L.t.s structure on $L$ given by the bracket $[\c,\c,\c]=[[\c,\c],\c]$. Now,
for all $x,y,z\in L$ and $u,v,w\in V$, we have 
\begin{align*}
    [x+u,y+v,z+w]_{L\oplus V}&=[[x+u,y+v]_{L\oplus V},z+w]_{L\oplus V}\\
    &=[[x,y]+\rho(x)(v)-\rho(y)(u),z+w]_{L\oplus V}\\
    &=[[x,y],z]+\rho([x,y])(w)-\rho(z)\rho(x)(v)+\rho(z)\rho(y(u)\\
    &=[x,y,z]+\Big(\rho(x)\rho(y)-\rho(y)\rho(x)\Big)(w)-\rho(z)\rho(x)(v)+\rho(z)\rho(y(u)\\
    &=[x,y,z]+\Big(\theta_{\rho}(y,x)-\theta_{\rho}(x,y)\Big)(w)-\theta_{\rho}(x,z)(v)+\theta_{\rho}(y,z)(u)\\
    &=[x,y,z]+D_{\rho}(x,y)(w)-\theta_{\rho}(x,z)(v)+\theta_{\rho}(y,z)(u).
\end{align*}
Then by Eq.\eqref{semidirect product}, $(L, [\c,\c,\c]=[\c,\c]\circ ([\c,\c]\otimes Id_L);\theta_{\rho})$ is a \textsf{L.t.sRep} pair.
\end{proof}
\begin{thm}
Every $1$-cocycle for the cohomology of LieRep pair $(L,[\c,\c];\rho)$ is a $1$-cocycle for the cohomology of the \textsf{L.t.sRep} pair $(L, [\c,\c,\c]=[\c,\c]\circ ([\c,\c]\otimes Id_L);\theta_{\rho})$.
\end{thm}
\begin{proof}
Let $\varphi$ be a $1$-cocycle of the cohomology of the LieRep pair $(L,[\c,\c],\rho)$, then 
\begin{equation*}
    \forall x,y\in L,\;\partial_{\rho}(\varphi)(x,y)=\rho(x)\varphi(y)-\rho(y)\varphi(x)-\varphi([x,y])=0.
\end{equation*}
Now, for any $x,y,z\in L$, we have
\begin{align*}
    \delta^{1}(\varphi)(x,y,z)=& D_{\rho}(x,y)\varphi(z)-\theta_{\rho}(x,z)\varphi(y)+\theta_{\rho}(y,z)\varphi(x)-\varphi([x,y,z])\\
    =&\rho([x,y])\varphi(z)-\rho(z)\rho(x)\varphi(y)+\rho(z)\rho(y)\varphi(x)-\varphi([[x,y],z])\\
    =&\rho([x,y])\varphi(z)-\rho(z)\rho(x)\varphi(y)+\rho(z)\rho(y)\varphi(x)-\rho([x,y])\varphi(z)+\rho(z)\varphi([x,y])\\
    =&\rho([x,y])\varphi(z)-\rho(z)\rho(x)\varphi(y)+\rho(z)\rho(y)\varphi(x)-\rho([x,y])\varphi(z)\\&+\rho(z)\rho(x)\varphi(y)-\rho(z)\rho(y)\varphi(x)\\=&0,
\end{align*}
which means that $\varphi$ is a $1$-cocycle for the  cohomology of the \textsf{L.t.sRep} pair $(L, [\c,\c,\c]=[\c,\c]\circ ([\c,\c]\otimes Id_L);\theta_{\rho})$.
\end{proof}
\begin{thm}
Let $\varphi\in Z^{2}_{Lie}(L,V)$. Then $\omega(x,y,z)=\varphi([x,y],z)-\rho(z)\varphi(x,y)$ is a $3$-cocycle of the \textsf{L.t.sRep} pair $(L, [\c,\c,\c]=[\c,\c]\circ ([\c,\c]\otimes Id_L);\theta_{\rho})$.
\end{thm}
\begin{proof}
Let $\varphi\in Z^{2}_{Lie}(L,V)$.  According to \eqref{SkewCochain} in Defintion \ref{cocycles}, we have
\begin{align*}
  &\omega(x,y,z)=\varphi([x,y],z)-\rho(z)\varphi(x,y)\\
  &\quad \;\;\;\;\;\; \; \quad=-\Big(\varphi([y,x],z)-\rho(z)\varphi(y,z)\Big)=-\omega(y,x,z).
\end{align*}
and 
\begin{align*}
  &\omega(x,y,z)+\omega(y,z,x)+\omega(z,x,y)\\ =& \varphi([x,y],z)-\rho(z)\varphi(x,y)+\varphi([y,z],x)-\rho(x)\varphi(y,z)+\varphi([z,x],y)-\rho(y)\varphi(z,x)\\
  =&\partial_{\rho}(\varphi)(x,y,z)=0.
\end{align*}
Now, for any $x,y,z,t,e\in L$, we have
\begin{align*}
    &\delta^{3}(\omega)(x,y,z,t,e)\\
    =&\omega(x,y,[z,t,e])+D_{\rho}(x,y)\omega(z,t,e)-\omega([x,y,z],t,e)-\omega(z,[x,y,t],e)\\
    &-\omega(z,t,[x,y,e])-\theta_{\rho}(t,e)\omega(x,y,z)+\theta_{\rho}(z,e)\omega(x,y,t)-D_{\rho}(z,t)\omega(x,y,e)\\
    =&\varphi([x,y],[[z,t],e])-\rho([[z,t],e])\varphi(x,y)+\rho([x,y])\varphi([z,t],e)-\rho([x,y])\rho(e)\varphi(z,t)\\
    &-\varphi([[[x,y],z],t],e)+\rho(e)\varphi([[x,y],z],t)-\varphi([z,[[x,y],t]],e)+\rho(e)\varphi(z,[[x,y],t])\\
    &-\varphi([z,t],[[x,y],e])+\rho([[x,y],e])\varphi(z,t)-\rho(e)\rho(t)\varphi([x,y],z)+\rho(e)\rho(z)\varphi([x,y],t)\\
    &+\rho(e)\rho([t,z])\varphi(x,y)-\rho([z,t])\varphi([x,y],e)+\rho([z,t])\rho(e)\varphi(x,y)\\
    =&\varphi([x,y],[[z,t],e])-\varphi([z,t],[[x,y],e])+\varphi([[x,y],[t,z]],e)+\rho([x,y])\varphi([z,t],e)\\
    &-\rho([z,t])\varphi([x,y],e)+\rho(e)\varphi([[x,y],z],t)+\rho(e)\varphi(z,[[x,y],t])-\rho(e)\rho(t)\varphi([x,y],z)\\
    &+\rho(e)\rho(z)\varphi([x,y],t)-\rho(e)\rho([x,y])\varphi(z,t)\\
    =&\varphi([x,y],[[z,t],e])-\varphi([z,t],[[x,y],e])+\varphi([[x,y],[t,z]],e)+\rho([x,y])\varphi([z,t],e)\\
    &-\rho([z,t])\varphi([x,y],e)+\rho(e)\varphi([x,y],[z,t]).
\end{align*}
Since $\varphi\in Z^{2}_{Lie}(L,V)$, we get $\delta^{3}(\omega)(x,y,z,t,e)=0$.
\end{proof}
\begin{lem}\label{associated}
Let $\alpha\in C^{1}_{Lie}(L,V)$. Then
\begin{equation}
    \delta^{1}(\alpha)(x,y,z)=\partial_{\rho}(\alpha)([x,y],z)-\rho(z)\partial_{\rho}(\alpha)(x,y).
\end{equation}
\end{lem}
\begin{proof}
For any $x,y,z\in L$, we have
\begin{align*}
    &D_{\rho}(x,y)\alpha(z)-\theta_{\rho}(x,z)\alpha(y)+\theta_{\rho}(y,z)\alpha(x)-\alpha([x,y,z])\\
    =&\rho([x,y])\alpha(z)-\rho(z)\rho(x)\alpha(y)+\rho(z)\rho(y)\alpha(x)-\alpha([[x,y],z])\\
    =&\rho([x,y])\alpha(z)-\rho(z)\alpha([x,y])-\alpha([[x,y],z])\\
    &-\rho(z)\Big(\rho(x)\alpha(y)-\rho(y)\alpha(x)-\alpha([x,y])\Big)\\
   = &\partial_{\rho}(\alpha)([x,y],z)-\rho(z)\partial_{\rho}(\alpha)(x,y).
\end{align*}
\end{proof}
\begin{pro}
Let $\varphi_1,\varphi_2\in Z^{2}_{Lie}(L,V)$. If $\varphi_1,\varphi_2$ are in the same cohomology class then $\omega_1,\omega_2$ defined by:
\begin{equation}
    \omega_i(x,y,z)=\varphi_i([x,y],z)-\rho(z)\varphi_i(x,y),\;i=1,2
\end{equation}
are in the same cohomology class of the associated L.t.s.
\end{pro}
\begin{proof}
Let $\varphi_1,\varphi_2\in Z^{2}_{Lie}(L,V)$ be two cocycles in the same cohomology class, that is
\begin{equation*}
   \varphi_2-\varphi_1=\partial_{\rho}(\alpha),\;\alpha\in C^{1}_{Lie}(L,V). 
\end{equation*}
and 
\begin{equation*}
    \omega_i(x,y,z)=\varphi_i([x,y],z)-\rho(z)\varphi_i(x,y),\;i=1,2.
\end{equation*}
According to Lemma \ref{associated},  we have
\begin{align*}
    &\omega_2(x,y,z)-\omega_1(x,y,z)=(\varphi_2-\varphi_1)([x,y],z)-\rho(z)(\varphi_2-\varphi_1)(x,y)\\
    &\quad \quad \quad \quad \quad \quad \quad \quad \quad\;\;\; =\partial_{\rho}(\alpha)([x,y],z)-\rho(z)\partial_{\rho}(x,y)\\
    &\quad \quad \quad \quad \quad \quad \quad \quad \quad\;\;\;=\delta^{1}(\alpha)(x,y,z),\;\alpha\in C^{1}_{Lie}(L,V),
\end{align*}
which means that $\omega_1$ and $\omega_2$ are in the same cohomology class.
\end{proof}
\begin{pro}
Let $T:V \rightarrow L$ be an $\mathcal{O}$-operator on a LieRep pair $(L,[\c,\c];\rho)$. Then $T$ is also an $\mathcal{O}$-operator on the \textsf{L.t.sRep} pair $(L, [\c,\c,\c]=[\c,\c]\circ ([\c,\c]\otimes Id_L);\theta_{\rho})$.
\end{pro}
\begin{proof}
For any $u,v,w\in V$, we have
\begin{align*}
    &[Tu,Tv,Tw]=[[Tu,Tv],Tw]=[T(\rho(Tu)v-\rho(Tv)u),Tw]\\
    &\quad \quad \quad \quad\; \;\;\;\; =[T\rho(Tu)v,Tw]-[T\rho(Tv)u,Tw]\\
    &\quad \quad \quad \quad\; \;\;\;\; =T\Big(\rho(T\rho(Tu)v)w-\rho(Tw)\rho(Tu)v\Big)-T\Big(\rho(T\rho(Tv)u)w-\rho(Tw)\rho(Tv)u\Big)\\
    &\quad \quad \quad \quad\; \;\;\;\;=T\Big(\rho([Tu,Tv])w-\rho(Tw)\rho(Tu)v +\rho(Tw)\rho(Tv)u\Big)\\
    &\quad \quad \quad \quad\; \;\;\;\;=T\Big(D_{\rho}(Tu,Tv)w+\theta_{\rho}(Tv,Tw)u-\theta_{\rho}(Tu,Tw)v\Big).
\end{align*}

Hence $T$ is a $\mathcal{O}$-operator on the \textsf{L.t.sRep} pair $(L, [\c,\c,\c]=[\c,\c]\circ ([\c,\c]\otimes Id_L);\theta_{\rho})$.
\end{proof}
Now there are two methods of constructing a L.t.s structure on $V$ from a LieRep pair $(L, [\c, \c]; \rho)$.
On the one hand, we induce a L.t.s structure $(V, [\c, \c, \c]_T )$ of the $\mathcal{O}$-operator $T$ on
the corresponding  \textsf{L.t.sRep} pair $(L, [\c,\c,\c]=[\c,\c]\circ ([\c,\c]\otimes Id_L);\theta_{\rho})$, where
\begin{equation}
[u, v, w]_T
=D_{\rho}(T u, Tv)w + \theta_{\rho}(Tv, Tw)u
-\theta_{\rho}(T u, Tw)v.
\end{equation}
On the other hand, we firstly induce a Lie algebra $(V, [\c, \c]_T )$ of $T$ on the LieRep pair $(L, [\c, \c] ;\rho)$,
and then we give a L.t.s structure induced from the Lie algebra $(V, [\c, \c]_T )$ by the bracket $[\c,\c,\c]_T=[\c,\c]_T\circ ([\c,\c]_T\otimes Id_V)$. These two methods give us the same L.t.s structure on $V$.\\

The following results hold for the cohomology groups of an $\mathcal{O}$-operator and here we omit the proofs.
\begin{cor}
Every $1$-cocycle for the cohomology of an $\mathcal{O}$-operator $T$ on a LieRep  $(L,[\c,\c];\rho)$  is a $1$-cocycle for the cohomology of an $\mathcal{O}$-operator on the  \textsf{L.t.sRep} pair $(L, [\c,\c,\c]=[\c,\c]\circ ([\c,\c]\otimes Id_L);\theta_{\rho})$.
\end{cor}
\begin{cor}
Let $\varphi\in Z^{2}_{T}(V,L)$ be a $2$-cocycle on the cohomology of an $\mathcal{O}$-operator $T$ on a LieRep  $(L,[\c,\c];\rho)$. Then $\omega(u,v,w)=\varphi([u,v]_T,w)-\rho_T(w)\varphi(u,v)$ is a $3$-cocycle of the cohomology of  an $\mathcal{O}$-operator $T$ on the \textsf{L.t.sRep} pair $(L, [\c,\c,\c]=[\c,\c]\circ ([\c,\c]\otimes Id_L);\theta_{\rho})$.
\end{cor}

\end{document}